\documentclass{amsart}
\usepackage[T1]{fontenc}
\usepackage[utf8]{inputenc}
\usepackage[english]{babel}
\usepackage{lmodern}
\usepackage{amsmath,amssymb,amsthm,mathrsfs,yhmath}
\usepackage{todonotes}
\usepackage{xcolor}
\usepackage{hyperref}
\definecolor{darkblue}{RGB}{0,0,160}
\hypersetup{
  colorlinks=true,
  citecolor=darkblue,
  filecolor=black,
  linkcolor=darkblue,
  urlcolor=darkblue,
  pdftitle={A fully faithful p-adic Riemann--Hilbert functor for filtered coherent D-modules},
  pdfauthor={Finn Wiersig}
}

\theoremstyle{plain}
\newtheorem{thm}{Theorem}[section]
\newtheorem{lem}[thm]{Lemma}

\newtheorem{cor}[thm]{Corollary}
\newtheorem{prop}[thm]{Proposition}

\newtheorem{introthm}{Theorem}[section]

\theoremstyle{definition}
\newtheorem{defn}[thm]{Definition}

\newtheorem{construction}[thm]{Construction}

\theoremstyle{remark}

\newtheorem{remark}[thm]{Remark}

\numberwithin{equation}{section}

\makeatletter
\def\cal@symb#1|#2{\expandafter\def\csname #2#1\endcsname{\mathcal{#1}}}
\def\calsymbols#1#2{\@for\@tmpz:=#2\do{\expandafter\cal@symb\@tmpz|{#1}}}
\def\bB@symb#1|#2{\expandafter\def\csname #2#1\endcsname{\mathbb{#1}}}
\def\bBsymbols#1#2{\@for\@tmpz:=#2\do{\expandafter\bB@symb\@tmpz|{#1}}}
\def\bold@symb#1|#2{\expandafter\def\csname #2#1\endcsname{\mathbf{#1}}}
\def\boldsymbols#1#2{\@for\@tmpz:=#2\do{\expandafter\bold@symb\@tmpz|{#1}}}
\def\scr@symb#1|#2{\expandafter\def\csname #2#1\endcsname{\mathscr{#1}}}
\def\scrsymbols#1#2{\@for\@tmpz:=#2\do{\expandafter\scr@symb\@tmpz|{#1}}}
\def\frak@symb#1|#2{\expandafter\def\csname #2#1\endcsname{\mathfrak{#1}}}
\def\fraksymbols#1#2{\@for\@tmpz:=#2\do{\expandafter\frak@symb\@tmpz|{#1}}}

\def\dmth@p#1|{\expandafter\let\csname#1\endcsname\relax
  \expandafter\DeclareMathOperator\csname#1\endcsname{#1}}
\def\operators#1{\@for\@tmpz:=#1\do{\expandafter\dmth@p\@tmpz|}}
\makeatother
\calsymbols{c}{A,B,C,D,E,F,G,H,I,J,K,L,M,N,O,P,Q,R,S,T,U,V,W,X,Y,Z}
\bBsymbols{b}{A,B,C,D,E,F,G,H,I,J,K,L,M,N,O,P,Q,R,S,T,U,V,W,X,Y,Z}
\boldsymbols{bf}{r,q,A,B,C,D,E,F,G,H,I,J,K,L,M,N,O,P,Q,R,S,T,U,V,W,X,Y,Z}
\scrsymbols{s}{A,B,C,D,E,F,G,H,I,J,K,L,M,N,O,P,Q,R,S,T,U,V,W,X,Y,Z}
\fraksymbols{fr}{a,b,c,d,e,f,g,h,i,j,k,l,m,n,o,p,q,r,s,t,u,v,w,x,y,z,F,M,O,R,X,Z}

\operators{sup,Gal,dR,pdR,Hom,Cond,CondAb,Solid,Sp,Der,Spa,Pro,inf,cycl,id,H,perf,b,Ext,Top,uSolid,Vect,ProVect,Fil,cond,op,Sh,Psh,sh,coh,gr,Hdg,IndFil,Fil,Ban,IndBan,Ind,Sym,Spec,Supp,Char,Mod,Gr,im}
\renewcommand{\>}{\rangle}

\newcommand{\defas}{\mathrel{\mathop{:}}=}
\newcommand{\cal}[1]{\mathcal{#1}}

\DeclareMathOperator{\et}{\text{\'{e}t}} %
\DeclareMathOperator{\proet}{\text{pro\'{e}t}} %

\newcommand{\isomap}{\xrightarrow{\cong}}
\newcommand{\Dcap}{\mathcal{\wideparen{D}}}
\newcommand{\OB}{\mathcal{O}\kern-1pt\mathbb{B}}
\newcommand{\SolB}{\mathcal{S}\kern-1pt\textit{ol}\kern1pt\mathbb{B}}
\newcommand{\Ovartheta}{\mathcal{O}\kern-1pt\vartheta}
\newcommand{\Otheta}{\mathcal{O}\kern-1pt\theta}
\DeclareMathOperator{\Sol}{\mathcal{S}\kern-1pt\textit{ol}}
\DeclareMathOperator{\Rec}{\mathcal{R}\kern-1pt\textit{ec}}
\DeclareMathOperator{\intHom}{\underline{Hom}}

\DeclareMathOperator{\shHom}{\underline{\mathcal{H}\kern-.5pt\textit{om}}}

\begin{document}
\raggedbottom

\title[A fully faithful $p$-adic Riemann--Hilbert functor]{A fully faithful $p$-adic Riemann--Hilbert functor for filtered coherent $\cD$-modules}

\author{Finn Wiersig}
\address{National University of Singapore}
\email{fwiersig@nus.edu.sg}
\date{}

\begin{abstract}
  We prove that the solution functor from bounded good filtered
  complexes of coherent $\mathcal D$-modules on a smooth rigid-analytic variety to
  modules over the positive almost de Rham period sheaf is fully faithful.
  On filtered vector bundles with flat
  connection satisfying Griffiths transversality, the construction recovers, up to duality and extension of
  scalars, Scholze's horizontal sections functor.  
  As an application, we obtain a notion of de Rham Zariski-constructible sheaves,
  generalising the classical notion of de Rham local systems.
\end{abstract}

\maketitle

\tableofcontents

\section{Introduction}

\subsection{Infinite-order operators}

The Riemann--Hilbert correspondence originates in Hilbert's twenty-first
problem~\cite{Equationsdifferentiellesapointssinguliersreguliers,KashiwaraKawaiholonomicIII,Kashiwaraconstructibilityholonomic,KashiwaraRHforholonomicsystems,Mebkhoutuneequivalence}
and has become one of the central bridges between analysis, algebraic geometry,
and representation theory. Its applications range from Saito's theory of mixed
Hodge modules~\cite{MR1000123,MR1047415} to the proofs of the
Kazhdan--Lusztig conjecture~\cite{MR632980,MR610137}. In its classical form,
due independently to
Kashiwara~\cite{Kashiwaraconstructibilityholonomic,KashiwaraRHforholonomicsystems}
and Mebkhout~\cite{Mebkhoutuneequivalence}, the Riemann--Hilbert correspondence
identifies the opposite of the bounded derived category of regular holonomic
modules over the sheaf $\cD$ of finite-order differential operators on a smooth
complex algebraic variety with the bounded derived category of algebraically constructible sheaves of
complex vector spaces on the analytification of the algebraic variety.

There is also a purely analytic full-faithfulness theorem for modules over
infinite-order differential operators. Let $Y$ be a complex manifold, and let
$\cD^{\infty}$ denote the sheaf of infinite-order differential operators on
$Y$. Prosmans--Schneiders~\cite{PS00} proved that the solution functor
\begin{equation*}
  \Sol_{\cD^{\infty}}\colon
  D_{\perf}^{b}\left(\cD^{\infty}\right)^{\op}
  \hookrightarrow
  D\left(\bC_Y\right),
  \qquad
  \cal M^{\bullet}
  \mapsto
  R\shHom_{\cD^{\infty}}
  \left(
    \cal M^{\bullet},
    \cal O_Y
  \right)
\end{equation*}
is a fully faithful embedding of triangulated categories. Here
$D_{\perf}^{b}\left(\cD^{\infty}\right)$ denotes the category of bounded
perfect complexes of $\cD^{\infty}$-modules, and $\bC_Y$ is the constant
sheaf with value $\bC$.

Fix a prime number $p$. In~\cite{Wiersig2024reconstruction}, we proved a
$p$-adic analogue of the theorem of Prosmans--Schneiders. Let $k$ be a
complete discretely valued field of mixed characteristic $(0,p)$ with perfect
residue field, and let $X$ be a smooth rigid-analytic $k$-variety. In the
$p$-adic setting, sheaves of infinite-order differential operators arise
naturally from the theory of arithmetic $\cD$-modules initiated by
Berthelot~\cite{BerthelotDmodulesArithmeticI} and further developed by Huyghe
and collaborators~\cite{HuygheDdaggerAffiniteProjectif,NootHuygheBB,HuygheSchmidtStrauchArithmeticStructures}.
In the rigid-analytic setting used here, we work with the sheaf $\Dcap$ of
completed differential operators constructed by
Ardakov--Wadsley~\cite{AW19,MR3846550}.

The natural setting for the $p$-adic solution theory is the pro-\'etale site
$X_{\proet}$, with its projection $\nu\colon X_{\proet}\to X_{\et}$. Indeed, period
sheaves and perfectoid covers are essential for constructing a sufficiently
rich theory of horizontal sections for $p$-adic differential equations.
In~\cite{Wiersig2024reconstruction}, the coefficient sheaf is the
overconvergent almost de Rham period sheaf $\bB_{\pdR}^{\dagger}$ on
$X_{\proet}$. The bridge between $\Dcap$-modules and
$\bB_{\pdR}^{\dagger}$-modules is the
$\nu^{-1}\Dcap$-$\bB_{\pdR}^{\dagger}$-bimodule
$\OB_{\pdR}^{\dagger}$. The resulting solution functor
\begin{equation*}
  \Sol_{\Dcap}^{\dagger}\colon
  D_{\perf}^{b}\left(\Dcap\right)^{\op}
  \hookrightarrow
  D\left(\bB_{\pdR}^{\dagger}\right),
  \qquad
  \cal M^{\bullet}
  \mapsto
  R\shHom_{\nu^{-1}\Dcap}
  \left(
    \nu^{-1}\cal M^{\bullet},
    \OB_{\pdR}^{\dagger}
  \right)
\end{equation*}
is fully faithful by~\cite[Theorem A]{Wiersig2024reconstruction}; that paper
also treats the larger class of $\cal C$-complexes.

\subsection{Main result}

The purpose of the present paper is to show that, in the $p$-adic setting, one
has an analytic full-faithfulness theorem for modules over \emph{finite}-order differential
operators.

This stands in striking contrast to the theorem of Prosmans--Schneiders, which
is formulated for the sheaf $\cD^{\infty}$ of infinite-order differential
operators rather than for the sheaf $\cD$ of finite-order differential
operators. The classical Riemann--Hilbert correspondence of Kashiwara and
Mebkhout does concern finite-order $\cD$-modules, but it is a different kind
of statement: it imposes the much stronger finiteness condition of regular
holonomicity and takes values in constructible sheaves. The theorem proved here
is instead a full-faithfulness theorem for bounded coherent complexes of
rigid-analytic $\cD$-modules, with values in a category of
period sheaves.

Let $\cD$ denote the sheaf of finite-order
differential operators on $X$, equipped with its order filtration.
The category of sheaves of filtered $\cD$-modules is \emph{quasi-abelian}
in the sense of Schneiders~\cite{Sch99}. This allows us to consider its
derived category, which we denote by $D_{F}\left(\mathcal D\right)$.

The coefficient sheaf in our solution
functor is the positive almost de Rham period sheaf $\bB_{\pdR}^{+}$ on $X_{\proet}$. At
the level of absolute period rings, it is obtained by adjoining a formal
logarithm of Fontaine's period $t$: $B_{\pdR}^{+}:=B_{\dR}^{+}[\log t]$.
We will also use the almost de Rham period sheaf $\bB_{\pdR}$,
which we obtain from $\bB_{\pdR}^{+}$ by inverting $t$.
Our reason for working with $\bB_{\pdR}$ rather than $\bB_{\dR}$
is that over $\bB_{\pdR}$, certain extensions vanish.

The corresponding period structure sheaf is
$\OB_{\pdR}
  :=
  \bB_{\pdR}
  \widehat\otimes_{\bB_{\dR}}
  \OB_{\dR}$.
Here, $\OB_{\dR}$ denotes the de Rham period structure sheaf
on $X_{\proet}$. $\OB_{\pdR}$ carries the induced Hodge filtration.

\begin{introthm}[Corollary~\ref{cor:fullyfaithful-unfiltered-recII}]
\label{introthm:main}
  The filtered solution functor
  \begin{equation*}
    \Sol\colon
    D_{F,\coh}^{b}\left(\cD\right)^{\op}
    \hookrightarrow
    D\left(\bB_{\pdR}^{+}\right),
    \qquad
    \cal N^{\bullet}
    \mapsto
    F_0R\shHom_{\nu^{-1}\cD}^{F}
    \left(
      \nu^{-1}\cal N^{\bullet},
      \OB_{\pdR}
    \right)
  \end{equation*}
  is a fully faithful embedding of triangulated categories
  on the derived category of filtered $\cD$-modules whose cohomology sheaves
  carry good filtrations.
\end{introthm}

\begin{remark}
A filtration $F_\bullet\cal N$ on a coherent $\cD$-module is
\emph{good} if it is exhaustive, locally bounded below, compatible with the
order filtration,
and its Rees module is coherent over the Rees algebra $\sR\cD$.
\end{remark}

\begin{remark}\label{rmk:introthmmain-mic}
  Let $(\cal E,\nabla,F_\bullet)$ be a filtered vector bundle with
  integrable connection satisfying Griffiths transversality in the sense
  of Scholze~\cite{Sch13pAdicHodge}. Viewed as a filtered $\cD$-module,
  $\cal E$ carries a good filtration. If $\cal E^\vee$ denotes the dual
  filtered connection, the filtered Spencer resolution gives a canonical
  identification
  \begin{equation*}
    \Sol(\cal E^\vee)
    \isomap
    F_0\left(
      \nu^{-1}\cal E
      \widehat\otimes_{\nu^{-1}\cO_X}
      \OB_{\pdR}
    \right)^{\nabla=0}.
  \end{equation*}
  The corresponding degree-zero $\bB_{\dR}^+$-lattice is
  \begin{equation*}
    F_0\left(
      \nu^{-1}\cal E
      \widehat\otimes_{\nu^{-1}\cO_X}
      \OB_{\dR}
    \right)^{\nabla=0}.
  \end{equation*}
  This agrees with Scholze's $\bB_{\dR}^+$-local system from
  \cite[Theorem~7.6]{Sch13pAdicHodge}. Thus, up to duality and extension
  of scalars, $\Sol$ recovers Scholze's construction on filtered vector
  bundles with integrable connection, while
  Theorem~\ref{introthm:main} extends full faithfulness to good filtered
  coherent $\cD$-complexes.
\end{remark}

\subsection{Application: de Rham Zariski-constructible sheaves}

Recall that an étale $\bZ_p$-local system $\bL$ on $X_{\et}$ is de Rham in the
sense of~\cite[Definition 8.3]{Sch13pAdicHodge} and~\cite{LiuZhuRH2017} if there exist a filtered
vector bundle with integrable connection
$(\cal E,\nabla,F_\bullet)$ on $X$, satisfying Griffiths
transversality, and a filtered isomorphism
\begin{equation*}
  F_0\left(
    \nu^{-1}\cal E
    \widehat\otimes_{\nu^{-1}\cO_X}
    \OB_{\dR}
  \right)^{\nabla=0}
  \otimes_{\bB_{\dR}^{+}}\bB_{\dR}
  \cong
  \widehat{\bL}\otimes_{\widehat{\bZ}_p}\bB_{\dR}.
\end{equation*}
The left-hand side identifies with
$\SolB_{\dR}^{+}(\cal E^\vee)\widehat{\otimes}_{\bB_{\dR}^{+}}^{L}\bB_{\dR}$, where
\begin{equation*}
  \SolB_{\dR}^{+}\colon
  D_{F,\coh}^{b}(\cD)^{\op}
  \to
  D(\bB_{\dR}^{+}),
  \qquad
  \cal N^\bullet
  \mapsto
  F_0R\shHom_{\nu^{-1}\cD}^{F}
  \left(
    \nu^{-1}\cal N^\bullet,
    \OB_{\dR}
  \right).
\end{equation*}
This suggests an extension of de Rhamness to the
Zariski-constructible complexes introduced by
Bhatt--Hansen~\cite{BhattHansenZC}. In
\S\ref{section:application--dR}, we call a Zariski-constructible complex
$\cal K^\bullet\in D_{\mathrm{zc}}^b(X_{\et},\bZ_p)$ \emph{de Rham} if there exist
$\cal N^\bullet\in D_{F,\coh}^{b}(\cD)$ and an isomorphism
\begin{equation*}
  \SolB_{\dR}^{+}(\cal N^\bullet)
  \widehat{\otimes}_{\bB_{\dR}^{+}}^L
  \bB_{\dR}
  \cong
  \cal K_{\proet}^\bullet
  \widehat{\otimes}_{\widehat{\bZ}_p}^L
  \bB_{\dR}
\end{equation*}
in the derived category $D_F(\bB_{\dR})$ of filtered $\bB_{\dR}$-modules;
$\cal K_{\proet}^\bullet$ denotes the pro-étale realisation of $\cal K^\bullet$.

For an étale $\bZ_p$-local system, the classical de Rham condition immediately
implies the new one. The converse is formulated in
Theorem~\ref{thm:comparison-Scholze-deRham}, and proved using
Theorem~\ref{introthm:main}.
Thus Theorem~\ref{thm:comparison-Scholze-deRham} shows that the two notions agree on
étale $\bZ_p$-local systems.

\subsection{Proof of Theorem~\ref{introthm:main}}

The proof follows the Morita-theoretic strategy of
Prosmans--Schneiders~\cite{PS00}. Standard adjunctions give the filtered reconstruction functor
\begin{equation*}
  \Rec\colon
  D(\bB_{\pdR}^{+})^{\op}
  \to
  D_{F}\left(\cD\right),
  \qquad
  \cal F^{\bullet}
  \mapsto
  R\nu_*
  R\shHom_{\bB_{\pdR}}^{F}
  \left(
    \bB_{\pdR}\otimes_{\bB_{\pdR}^{+}}^{L}\cal F^{\bullet},
    \OB_{\pdR}
  \right),
\end{equation*}
and the biduality morphism
$\rho_{\cal N^{\bullet}}\colon\cal N^{\bullet}\to
\Rec(\Sol(\cal N^{\bullet}))$;
$\bB_{\pdR}\otimes_{\bB_{\pdR}^{+}}\cal F^{\bullet}$
carries the filtration whose $n$th filtered piece is $F_n\bB_{\pdR}\otimes_{\bB_{\pdR}^{+}}\cal F^{\bullet}$.
All complexes of $\cD$-modules
which we consider are
locally perfect, so a standard argument reduces the theorem to
showing that
\begin{equation*}
  \rho_{\cD}\colon
  \cD
  \isomap
  R\nu_*
  R\shHom_{\bB_{\pdR}}^{F}
  \left(
    \OB_{\pdR},
    \OB_{\pdR}
  \right)
\end{equation*}
is an isomorphism (Theorem~\ref{thm:rho-iso}).
We prove this as follows.
Locally on a toric affinoid $X=\Sp A$, put
$\delta_i=z_i\partial_i$ and $Q_i=t\delta_i$. On the perfectoid toric cover,
the $Q_i$ are coordinate derivations, so every descent-compatible
endomorphism has a Taylor expansion
$\sum_{\alpha}b_\alpha Q^{[\alpha]}$. Arithmetic descent identifies
$b_\alpha$ with an element of $A t^{-|\alpha|}$. A uniform upper bound on the filtration degrees of the coefficients
forces $|\alpha|$ to be bounded, and hence the expansion
to be finite. Since $t^{-|\alpha|}Q^{[\alpha]}=\delta^{[\alpha]}$, PBW
identifies these endomorphisms with $\cD(X)$. More precisely, endomorphisms
sending $F_m$ into $F_{m+n}$ for every $m$ correspond exactly to
$\cD_{\leq n}(X)$. Thus the endomorphism theorem is a strict filtered
isomorphism, and applying the Rees construction gives the displayed
isomorphism. The same cohomology calculation gives the vanishing in higher
degrees, and the local identifications glue.

\subsection{Comparison with the solution functor for
\texorpdfstring{$\Dcap$}{completed differential operator}-modules}

In~\cite{Wiersig2024reconstruction}, we defined the codomain of the functor
$\Sol_{\Dcap}^{\dag}$ to be the derived category of sheaves of $\bB_{\pdR}^{\dag}$-modules,
where the sections carry additional topological data. We can enhance
and upgrade this
functor so that it sends a complex of filtered $\Dcap$-modules into
the derived category of sheaves of $\bB_{\pdR}^{\dag,+}$-modules,
together with topological data. Then, for
every complex $\cal M^{\bullet}$ of filtered $\cD$-modules there is a canonical morphism
\begin{equation*}
  \Xi_{\cal M^{\bullet}}\colon
    \left\vert\Sol_{\Dcap}^{\dagger}
    \left(\Dcap\widehat{\otimes}_{\cD}^{L}\cal M^{\bullet}\right)\right\vert
  \otimes_{\left\vert\bB_{\pdR}^{\dagger,+}\right\vert}^{L}\bB_{\pdR}^{+}
  \to
  \Sol\left(\cal M^{\bullet}\right)
\end{equation*}
in the derived category of sheaves of filtered $\bB_{\pdR}^{+}$-modules,
where $|\cdot|$ indicates that we forget the topologies, but retain the filtrations.

The $p$-adic Cauchy theorem proved in~\cite{WiersigCauchy} implies that
$\Xi_{\cal E}$ is an isomorphism for every vector bundle $\cal E$ with
integrable connection.
For general coherent complexes this remains open, already for holonomic
$\cD$-modules. By analogy with the complex theory
\cite[Theorem 6.4.1]{KashiwaraKawaiholonomicIII}, the cone of $\Xi$ should
measure the failure of overconvergence and may be related to an appropriate
notion of regularity.


\subsection{Acknowledgements}

I thank Tomoyuki Abe, Konstantin Ardakov, Andreas Bode, and David Hansen for helpful
conversations.

ChatGPT was used for language editing and to assist with checking
mathematical arguments and references. The author reviewed the
resulting suggestions and assumes full responsibility for the article.


\subsection{Conflict of interest}
The author declares that they have no conflict of interest.


\subsection{Ground fields and rings}
Let $k$ denote a complete discretely valued field of mixed characteristic
$(0,p)$ with perfect residue field $\kappa$.
Fix a uniformiser $\pi\in k$ and write $k^{\circ}\subset k$
for the ring of power-bounded elements. Set $k_{0}:=W(\kappa)[1/p]$.
Fix an algebraic closure $\overline{k}$ of $k$, write
$G_k:=\Gal(\overline{k}/k)$, and let $C$ be its completion.


\subsection{Conventions}

The natural numbers are $\bN=\{0,1,2,3,\dots\}$.
Modules are always left modules.

\section{Background}

\subsection{Filtrations}
\label{subsec:Fil}

A \emph{filtration} $F_\bullet R$ on a commutative ring $R$ is a family
$\{F_nR\colon n\in\bZ\}$ of additive subgroups of $R$ such that
$1\in F_0R$, $F_nR\subseteq F_mR$ whenever $n\leq m$, and
$F_nR\cdot F_mR\subseteq F_{n+m}R$ for all $n,m\in\bZ$. Thus, all
filtrations considered here are increasing. The filtration is called
\emph{separated} if $\bigcap_{n\in\bZ}F_nR=\{0\}$ and
\emph{exhaustive} if $\bigcup_{n\in\bZ}F_nR=R$. Unless stated otherwise,
all filtrations considered below are assumed to be exhaustive.

The filtration $F_\bullet R$ defines a topology on $R$ for which the
subgroups $F_nR$, with $n\in\bZ$, form a fundamental system of
neighbourhoods of $0$. With this topology, $R$ is a topological ring. The
filtration is called \emph{complete} if every Cauchy sequence converges to
a unique limit. Equivalently, the canonical morphism
$R\to\varprojlim_{n\to+\infty}R/F_{-n}R$ is an isomorphism of additive
groups. In particular, a complete filtration is separated.

Let $F_\bullet R$ be a filtered commutative ring and let $M$ be an
$R$-module. A \emph{filtration} $F_\bullet M$ on $M$ is a family
$\{F_nM\colon n\in\bZ\}$ of additive subgroups of $M$ such that
$F_nM\subseteq F_mM$ whenever $n\leq m$, and
$F_nR\cdot F_mM\subseteq F_{n+m}M$ for all $n,m\in\bZ$. As above, the
filtration is called \emph{separated} if
$\bigcap_{n\in\bZ}F_nM=\{0\}$ and \emph{exhaustive} if
$\bigcup_{n\in\bZ}F_nM=M$. It defines a topology on $M$ for which the
subgroups $F_nM$ form a fundamental system of neighbourhoods of $0$. The
filtration is called \emph{complete} if every Cauchy sequence in $M$
converges to a unique limit, or equivalently if the canonical morphism
$M\to\varprojlim_{n\to+\infty}M/F_{-n}M$ is an isomorphism.

A \emph{morphism of filtered $R$-modules $M\to N$}
is a morphism $f\colon M\to N$ of $R$-modules such that
$f\left(F_nM\right)\subseteq F_nN$ for all $n\in\bZ$.
We denote the sets of filtered $R$-module morphisms by $\Hom^{F}$.
The resulting category of filtered $R$-modules
\[
\Fil_R
\]
is a quasi-abelian category in the sense of
Schneiders~\cite{Sch99}, which provides a robust framework for homological
algebra and for sheaves with values in quasi-abelian categories.
Following the arguments in~\cite[\S3.1]{Sch99}, one checks
furthermore that $\Fil_R$ is elementary. In particular, the category
has enough projective objects and we may form its derived
category $D(\Fil_R)$ in the sense of~\cite{Sch99}.

The category $D(\Fil_R)$ carries the left $t$-structure
with heart $LH(\Fil_R)$. It is a Grothendieck abelian category, with enough projectives and
injectives. Schneiders~\cite[Proposition 3.1.28]{Sch99} computes that
$LH(\Fil_R)$ is canonically equivalent to the category of graded modules
over the Rees ring
\begin{equation*}
  \sR R
  :=
  \bigoplus_{n\in\mathbb Z}F_{n}R\hbar^n.
\end{equation*}
It follows from the abstract theory of quasi-abelian categories
that the Rees functor
\begin{equation*}
  I=\sR\colon\Fil_R\hookrightarrow \Gr_{\sR R},\quad  M\mapsto \sR M:=\bigoplus_nF_n M\hbar^n
\end{equation*}
induces an equivalence $D(\Fil_R)\simeq D(\Gr_{\sR R})$,
where $\Gr_{\sR R}$ denotes the category of graded $\sR R$-modules.
We pass between the two settings without comment when no confusion can arise.

The category $\Fil_R$ is closed symmetric monoidal with respect to the
usual filtered tensor product $\otimes_R$. We denote its
internal Hom by $\intHom_R^{F}$, so that
$F_{0}\intHom^{F}_R=\Hom^{F}_R$ and
$\otimes_R$ and $\intHom_R^{F}$ are related by the usual
tensor--Hom adjunction.

 A \emph{filtered $R$-algebra} $A$ is a monoid object in $\Fil_R$.
 A \emph{filtered $A$-module} $M$ is a module over the monoid object $A$.
 As above, we have the Rees ring $\sR A$ and the Rees functor $M\mapsto\sR M$.


Let $S$ be a site. A \emph{presheaf of filtered $R$-modules} on $S$ is
a functor $\cF\colon S^{\op}\to\Fil_R$. Such a presheaf is called a
\emph{sheaf of filtered $R$-modules} if, for every covering family
$\cU=\{U_i\to U\}_{i\in J}$, the sequence
$0\to\cF(U)\to\prod_{i\in J}\cF(U_i)\to
\prod_{i,j\in J}\cF(U_i\times_UU_j)$ is strictly exact in $\Fil_R$.
Here the final morphism is the difference of the two restriction maps
induced by the projections $U_i\times_UU_j\to U_i$ and
$U_i\times_UU_j\to U_j$.

The resulting category $\Sh(S,\Fil_R)$ is quasi-abelian. The inclusion
\begin{equation*}
  \Sh(S,\Fil_R)\hookrightarrow\Psh(S,\Fil_R)
\end{equation*}
has a left adjoint $\cF\mapsto\cF^{\sh}$, called sheafification.

The tensor product of two sheaves $\cF$ and $\cG$ is defined by
$\cF\otimes_R\cG:=
\bigl(U\mapsto\cF(U)\otimes_R\cG(U)\bigr)^{\sh}$. This equips
$\Sh(S,\Fil_R)$ with a symmetric monoidal structure. It is closed: for
every sheaf $\cG$, the functor
$\cF\mapsto\cF\otimes_R\cG$ admits a right adjoint, which
we denote by $\shHom_R^{F}(\cG,-)$.

The symmetric monoidal structure gives a notion of monoid objects in $\Sh(S,\Fil_R)$,
which we call \emph{sheaves of filtered $R$-algebras}. For any sheaf $\mathcal R$
of filtered $R$-algebras, we denote its category of (left) module objects by $\Fil_{\mathcal R}$.
We call objects in $\Fil_{\mathcal R}$ \emph{sheaves of filtered $\mathcal R$-modules}.

\begin{lem}\label{lem:completion-ghostadjunction-sheaves}
  Let $\cR$ be a sheaf of filtered
  commutative rings, and let $\cM,\cN$ be sheaves of
  filtered $\cR$-modules. Write $\widehat{\cM}$ for the
  sheafification of the sectionwise separated completion.
  Suppose that, for every $n\in\bZ$, the canonical map
  \[
    F_n\cN\isomap
    R\varprojlim_{m\geq1}
    \left(F_n\cN/F_{n-m}\cN\right)
  \]
  is an isomorphism. Then restriction induces an isomorphism
  \[
    R\shHom_{\cR}^{F}(\widehat{\cM},\cN)
    \isomap R\shHom_{\cR}^{F}(\cM,\cN).
  \]
\end{lem}

\begin{proof}
  Let $\iota$ denote the embedding of the category $\Fil_{\mathcal{R}}$ of sheaves of filtered
  $\mathcal R$-modules into the category $\Fil_{\mathcal{R}}^{\mathrm{pre}}$ of
  presheaves of filtered $\mathcal R$-modules. Denote its right derived functor by $R\iota$.
  
  Write $\Fil_{\mathcal{R}}^{\mathrm{pre}}\to\Fil_{\mathcal{R}}^{\mathrm{pre}}$
  $\cF\mapsto\widehat{\cF}^{\mathrm{pre}}$ for the sectionwise separated
  completion functor. It is exact, in the sense that it preserves strict exact sequences.
  Indeed, strict exact sequences remain exact after applying
  $F_n(-)/F_{n-m}(-)$, and the resulting inverse systems have
  surjective transition maps. Exactness therefore follows
  from the Mittag--Leffler argument, applied sectionwise.
  Therefore, $\cF\mapsto\widehat{\cF}^{\mathrm{pre}}$
  induces an exact endofunctor of $D(\Fil_{\mathcal{R}}^{\mathrm{pre}})$. Moreover, it is
  idempotent, and the canonical morphisms $\cF\to\widehat{\cF}^{\mathrm{pre}}$ exhibit
  $\widehat{(-)}^{\mathrm{pre}}$ as the localisation onto the full subcategory of presheaves
  whose sections are separated
  and complete. Consequently, the induced endofunctor
  of $D(\Fil_{\mathcal{R}}^{\mathrm{pre}})$ is a localisation functor.

  We claim that $R\iota\cN$ is fixed by this localisation.
  Derived inclusion $R\iota$ is computed by derived sections on each filtration piece.
  Choosing a complex
  $\cK^\bullet$ of filtered presheaves representing
  $R\iota\cN$, we have
  \[
    F_n\cK^\bullet(V)\cong R\Gamma(V,F_n\cN)
  \]
  for every object $V\in S$ and every $n\in\bZ$.
  Taking cones shows that
  \[
    F_n\cK^\bullet(V)/F_{n-m}\cK^\bullet(V)
    \cong
    R\Gamma\left(V,F_n\cN/F_{n-m}\cN\right).
  \]
  The transition maps of these quotient complexes are
  surjective in every degree, so their ordinary inverse
  limit computes their derived inverse limit. Consequently,
  \[
  \begin{aligned}
    F_n\left(\widehat{\cK^\bullet}^{\mathrm{pre}}\right)(V)
    &\cong R\varprojlim_m
      R\Gamma\left(V,F_n\cN/F_{n-m}\cN\right)\\
    &\cong R\Gamma\left(
      V,R\varprojlim_m F_n\cN/F_{n-m}\cN\right)\\
    &\cong R\Gamma(V,F_n\cN).
  \end{aligned}
  \]
  Here the last isomorphism is the hypothesis.
  These identifications show that the canonical map
  $R\iota\cN\to\widehat{R\iota\cN}^{\mathrm{pre}}$
  is an isomorphism.

  Since sheafification is exact and
  $\widehat{\cM}=(\widehat{\iota\cM}^{\mathrm{pre}})^{\sh}$,
  derived adjunction and the localisation property give
  \[
  \begin{aligned}
    R\Hom_{\Sh}^{F}(\widehat{\cM},\cN)
    &\cong R\Hom_{\Psh}^{F}
      \left(\widehat{\iota\cM}^{\mathrm{pre}},
      R\iota\cN\right)\\
    &\cong R\Hom_{\Psh}^{F}
      \left(\iota\cM,R\iota\cN\right)\\
    &\cong R\Hom_{\Sh}^{F}(\cM,\cN)
  \end{aligned}
  \]
  where the subscripts $\Sh$ and $\Psh$ indicate whether we are computing the
  Hom in the category of sheaves or presheaves, respectively.
  Applying the same argument on every slice of the site
  and to every filtration shift of $\mathcal{N}$ proves the
  assertion for derived internal Hom.
\end{proof}

\subsection{Finite-order differential operators on rigid-analytic spaces}

Let $X$ be a smooth rigid-analytic variety over $k$. If $Y=\Sp B$ is an
affinoid subdomain of $X$, we write
$\cT(Y):=\Der_k(B)$ for the $B$-module of continuous $k$-linear
derivations of $B$. Equipped with the commutator bracket and its natural
action on $B$, the module $\cT(Y)$ is a $(k,B)$-Lie algebra in the sense
of~\cite{MR154906}. We may therefore form its relative enveloping algebra
$U_B(\cT(Y))$.

We denote by $\cD$ the sheaf of finite-order $k$-linear differential
operators on $X$. On an affinoid subdomain $Y=\Sp B$, it is given by
$\cD(Y):=U_B(\cT(Y))$. The universal property of the relative
enveloping algebra provides a canonical homomorphism
$i_B\colon B\to\cD(Y)$ of $k$-algebras.

We equip both $k$ and every affinoid $k$-algebra $B$ with their
trivial increasing filtrations. Explicitly, we set $F_nk=k$ and $F_nB=B$
for $n\geq0$, and $F_nk=F_nB=0$ for $n<0$.

The algebra $\cD(Y)$ carries its increasing order filtration
$\cD_{\leq n}(Y)$, where $\cD_{\leq n}(Y)$ is the left $B$-submodule
generated by products of at most $n$ elements of $\cT(Y)$. We set
$\cD_{\leq -1}(Y):=0$. Thus
$\cD(Y)=\bigcup_{n\geq0}\cD_{\leq n}(Y)$ and
$\cD_{\leq r}(Y)\cD_{\leq s}(Y)
\subseteq\cD_{\leq r+s}(Y)$ for all $r,s\geq0$.

\begin{lem}\label{lem:D-local-description}
  Suppose that $X$ is equipped with an étale morphism
  \begin{equation*}
    X\to\bT^d=\Sp k\langle z_1^{\pm1},\dots,z_d^{\pm1}\rangle.
  \end{equation*}
  For every
  affinoid subdomain $Y=\Sp B$ of $X$, let
  $\partial_i\in\cT(Y)$ denote the unique lift of the coordinate vector
  field $d/dz_i$ on $\bT^d$. For
  $\alpha=(\alpha_1,\dots,\alpha_d)\in\bN^d$, write
  $\partial^\alpha:=\partial_1^{\alpha_1}\cdots
  \partial_d^{\alpha_d}$ and
  $|\alpha|:=\alpha_1+\cdots+\alpha_d$.

  Every element $P\in\cD(Y)$ admits a unique expression
  $P=\sum_{\alpha\in\bN^d}f_\alpha\partial^\alpha$, where
  $f_\alpha\in B$ and all but finitely many of the $f_\alpha$ vanish.
  Moreover, for every $n\geq0$, there is a functorial isomorphism
  $\cD_{\leq n}(Y)\cong
  \bigoplus_{|\alpha|\leq n}B\partial^\alpha$ of left $B$-modules.
\end{lem}

\begin{proof}
  Since $Y\to\bT^d$ is étale, the natural base-change morphism
  $B\otimes_{k\langle z_1^{\pm1},\dots,z_d^{\pm1}\rangle}
  \cT(\bT^d)\to\cT(Y)$ is an isomorphism. In particular,
  $\cT(Y)$ is a free $B$-module with basis
  $\partial_1,\dots,\partial_d$. The 
  PBW theorem for
  Lie--Rinehart algebras identifies
  $\gr\cD(Y)$ with $\Sym_B(\cT(Y))$. The ordered monomials
  $\partial^\alpha$ therefore form a basis of $\cD(Y)$ as a left
  $B$-module, and the monomials with $|\alpha|\leq n$ form a basis of
  $\cD_{\leq n}(Y)$.
\end{proof}

Restriction maps between affinoid subdomains preserve the order
filtration. Consequently, $Y\mapsto\cD(Y)$ defines a presheaf of
filtered $k$-algebras on the affinoid site of $X$. This presheaf is a
sheaf. Indeed, locally on a toric chart,
Lemma~\ref{lem:D-local-description} identifies each finite-order piece
$\cD_{\leq n}$ with a finite direct sum of copies of the structure
sheaf. Hence $\cD_{\leq n}$ is a sheaf of $k$-vector spaces for every
$n$. The assertion for $\cD$ then follows from the strict exactness of
filtered colimits in $\Fil_k$.

Similarly, let $\cO$ denote the structure sheaf endowed with the
trivial filtrations on affinoid sections, so that $\cO(Y)=B$ for
$Y=\Sp B$. The homomorphisms $i_B$ are compatible with restriction and
therefore define a canonical morphism
$\cO\to\cD$ of sheaves of filtered $k$-modules.

\subsection{The pro-étale site}

We identify $X$ with its associated adic space and denote Scholze's
pro-\'etale site by $X_{\proet}$; coverings are understood in the corrected
sense of~\cite{Sch13pAdicHodgeErratum}. There is a canonical morphism of
sites $\nu\colon X_{\proet}\to X_{\et}$.
The affinoid perfectoid objects form a basis of $X_{\proet}$
\cite[Definition 4.3 and Corollary 4.7]{Sch13pAdicHodge}. If such an object
$U$ has completed affinoid algebra $(R,R^+)$, we write
$\widehat U=\Spa(R,R^+)$.

Locally, choose an \'etale morphism $X\to\bT^d$. After base change to $C$,
adjoining all $p$-power roots of the toric coordinates gives an affinoid
perfectoid cover $\widetilde X_C\to X_{C}$. Its geometric Galois group is
$\bZ_p(1)^d$; this cover will be used throughout the local computations.

\subsection{Period sheaves}
\label{subsec:period-sheaves}

We recall the constructions of \cite[\S 6]{Sch13pAdicHodge}. All period
sheaves below are sheaves on $X_{\proet}$.

Fix a compatible system
$\epsilon=(1,\zeta_p,\zeta_{p^2},\dots)\in C^{\flat+}$, where
$\zeta_{p^n}$ is a primitive $p^n$-th root of unity. Let $(R,R^+)$ be a
perfectoid affinoid $C$-algebra and set
$R^{\flat+}:=\varprojlim_{x\mapsto x^p}R^+/p$. The \emph{relative
infinitesimal period ring} is
$\bA_{\inf}(R,R^+)\defas W(R^{\flat+})$, and we set
$\bB_{\inf}(R,R^+)\defas\bA_{\inf}(R,R^+)[1/p]$. Fontaine's map
$\theta\colon\bA_{\inf}(R,R^+)\to R^+$ is given by
$\sum_{n\geq0}[a_n]p^n\mapsto\sum_{n\geq0}a_n^\sharp p^n$.
After inverting $p$, it induces a map
$\theta\colon\bB_{\inf}(R,R^+)\to R$, which we denote by the same symbol.
The \emph{relative positive de Rham period ring} is the $\ker\theta$-adic
completion
\begin{equation*}
  \bB_{\dR}^+(R,R^+)
  :=\varprojlim_j\bB_{\inf}(R,R^+)/(\ker\theta)^j.
\end{equation*}
The element $t:=\log[\epsilon]$ converges in $\bB_{\dR}^+(R,R^+)$ and
generates $\ker\theta$. We define
\begin{equation*}
  \bB_{\dR}(R,R^+)
  :=\bB_{\dR}^+(R,R^+)[1/t].
\end{equation*}

In analogy with Fontaine's almost de Rham period ring
$B_{\pdR}=B_{\dR}[\log t]$, see~\cite{Fontaine2004Arithmetic}, we define the
\emph{relative almost de Rham period ring}
\begin{equation*}
  \bB_{\pdR}(R,R^+)
  :=\bB_{\dR}(R,R^+)[\log t],
\end{equation*}
where $\log t$ is a formal variable. The natural action of $G_k$ on $X_C$
induces semilinear actions on the period sheaves introduced below, and the
action on $\bB_{\pdR}$ is determined by
$\sigma(\log t)=\log t+\log\chi_{\cycl}(\sigma)$, where
$\chi_{\cycl}\colon G_k\to\bZ_p^\times$ is the cyclotomic character.
We will also consider the subring
\begin{equation*}
  \bB_{\pdR}^{+}(R,R^+):=\bB_{\dR}^{+}(R,R^+)[\log t].
\end{equation*}

We equip $k_0$ and $\bB_{\inf}(R,R^+)$ with the trivial filtrations.
The ring $\bB_{\dR}^+(R,R^+)$ carries the Hodge filtration
$F_{n}\bB_{\dR}^+(R,R^+):=(\ker\theta)^{-n}$ for $n\leq0$, and we set
$F_{n}\bB_{\dR}^+(R,R^+):=\bB_{\dR}^+(R,R^+)$ for $n\geq0$. Since
$t$ generates $\ker\theta$, this is the $t$-adic filtration. We extend it to
$\bB_{\dR}(R,R^+)$ by setting
$F_{n}\bB_{\dR}(R,R^+):=t^{-n}\bB_{\dR}^+(R,R^+)$ for $n\in\bZ$, and to
$\bB_{\pdR}(R,R^+)$ by setting
$F_{n}\bB_{\pdR}(R,R^+):=t^{-n}\bB_{\pdR}^+(R,R^+)$. Thus, $\log t$
has filtration degree zero. We do not equip $\bB_{\pdR}^{+}(R,R^+)$ with a filtration,
and simply view it as the zeroth filtered piece of $\bB_{\pdR}(R,R^+)$.

We refer to \cite[Definition 6.1]{Sch13pAdicHodge} for the construction of
the sheaves $\bB_{\dR}^+$ and $\bB_{\dR}$ on $X_{\proet}$. If
$U\in X_{\proet}$ is affinoid perfectoid and
$\widehat{U}=\Spa(R,R^+)$, then
$\bB_{\dR}^+(U)=\bB_{\dR}^+(R,R^+)$ and
$\bB_{\dR}(U)=\bB_{\dR}(R,R^+)$. We define $\bB_{\pdR}$ in the category
of sheaves of filtered $k$-algebras by
$\bB_{\pdR}(U):=\bB_{\dR}(R,R^+)[\log t]$,
whenever $U$ is an affinoid perfectoid over $X_{C}$, and for general
affinoid perfectoids by descent.
Similarly, we have the sheaf of abstract rings $\bB_{\pdR}^+(U):=\bB_{\dR}^+(R,R^+)[\log t]$.

\begin{lem}\label{lem:pdR-flat-dR}
  If $U\in X_{\proet}$ is affinoid perfectoid, then $\bB_{\pdR}(U)$ is flat
  as a $\bB_{\dR}(U)$-module with respect to the filtered tensor product.
\end{lem}

\begin{proof}
  In $\Fil_{\bB_{\dR}(U)}$ one has
  \begin{equation*}
    \bB_{\pdR}(U)
    \cong
    \bigoplus_{m\geq0}\bB_{\dR}(U)(\log t)^m.
  \end{equation*}
  The filtered tensor product commutes with coproducts, which preserve strict
  exact sequences. Hence tensoring with $\bB_{\pdR}(U)$ is exact.
\end{proof}


\subsection{Period structure sheaves}
\label{subsec:periodstructuresheaves}

We use the corrected definition of $\OB_{\dR}^+$ from
\cite{Sch13pAdicHodgeErratum}. Let
$U=\text{``}\varprojlim_{i\in I}\text{"}U_i\in X_{\proet}$ be affinoid
perfectoid, where $U_i=\Spa(R_i,R_i^+)$, and write
$\widehat{U}=\Spa(R,R^+)$. On the basis of affinoid perfectoid objects, first
consider the presheaf given by
\begin{equation*}
  U\mapsto
  \varinjlim_i\varprojlim_j
  \left(
    \left(R_i^+\widehat{\otimes}_{W(\kappa)}
    \bA_{\inf}(R,R^+)\right)[1/p]/(\ker\theta)^j
  \right).
\end{equation*}
Here, $\widehat{\otimes}$ denotes the $p$-adic completion of the tensor
product, and $\theta$ denotes the map
\begin{equation*}
  \left(R_i^+\widehat{\otimes}_{W(\kappa)}
  \bA_{\inf}(R,R^+)\right)[1/p]
  \to R,
  \quad
  f\widehat{\otimes}a\mapsto f\theta(a),
\end{equation*}
where $R_i^+\to R^+$ is the natural map. The sheaf $\OB_{\dR}^+$ is the
sheafification of this presheaf.

For $n\leq0$, we equip $\OB_{\dR}^+$ with the Hodge filtration
$F_n\OB_{\dR}^+:=(\ker\theta)^{-n}$, and we set
$F_n\OB_{\dR}^+:=\OB_{\dR}^+$ for $n\geq0$. We equip
$\OB_{\dR}^+(U)[1/t]$ with the filtration
\begin{equation}\label{eq:OBdR+t-1-HodgeFil}
  F_n\OB_{\dR}^+(U)[1/t]
  :=\sum_{m\in\bZ}t^{-m}F_{n-m}\OB_{\dR}^+(U).
\end{equation}
Following the convention of Diao--Lan--Liu--Zhu
\cite[Definition 2.2.10 and Remark 2.2.11]{MR4536903}, we define
$\OB_{\dR}(U)$ to be the separated completion of
$\OB_{\dR}^+(U)[1/t]$ with respect to this filtration. This construction is
compatible with restriction and defines a sheaf $\OB_{\dR}$ of
filtered $\bB_{\dR}$-algebras on $X_{\proet}$. We emphasise that
$\OB_{\dR}$ denotes the completed structural period sheaf here, rather than
Scholze's uncompleted sheaf $\OB_{\dR}^+[1/t]$.

Let $\cR$ be a sheaf of filtered rings. For any two sheaves $\cF$ and $\cG$ of filtered $\cR$-modules,
we write $\cF\widehat{\otimes}_{\cR}\cG$ for the sheafification
of the presheaf
\begin{equation*}
  U\mapsto\cF(U)\widehat{\otimes}_{\cR(U)}\cG(U).
\end{equation*}

\begin{prop}\label{prop:localdescription-of-OBdR+-sheaves-recII}
  Suppose that $X$ is affinoid and equipped with an \'etale morphism
  $X\to\bT^d=\Sp k\left\<z_1^{\pm},\dots,z_d^{\pm}\right\>$. Let
  $\widetilde{X}_C\to X_C$ be the associated standard toric pro-\'etale cover
  and write $T=(T_1,\dots,T_d)$. Equip $k_0[T]$ with the trivial filtration.
  Then there is an isomorphism
  \begin{equation}\label{eq:localdescription-of-OBdR+-sheaves-recII}
    \bB_{\dR}|_{\widetilde{X}_C}
    \widehat{\otimes}_{k_0}k_0[T]
    \isomap
    \OB_{\dR}|_{\widetilde{X}_C}
  \end{equation}
  of sheaves of filtered $\bB_{\dR}$-algebras on
  $X_{\proet}/\widetilde{X}_C$. Under this isomorphism,
  \begin{equation*}
    T_i
    \mapsto
    \frac{z_i\widehat{\otimes}1-
    1\widehat{\otimes}[z_i^\flat]}{t},
  \end{equation*}
  where $z_i^\flat=(z_i,z_i^{1/p},z_i^{1/p^2},\dots)$ on the standard toric
  cover.
\end{prop}

\begin{proof}
  We have to show that the morphism
  \begin{equation*}
    \bB_{\dR}(U)\widehat{\otimes}_{k_0}k_0[T]
    \to
    \OB_{\dR}(U)
  \end{equation*}
  is an isomorphism for every affinoid perfectoid
  $U\in X_{\proet}/\widetilde{X}_C$. The morphism is strict and compatible
  with the Hodge filtrations, and it induces an isomorphism on associated
  graded objects by \cite[Corollary 6.15]{Sch13pAdicHodge}. Since both sides
  are separated and complete for these filtrations, the morphism is an
  isomorphism. The construction is compatible with restriction in $U$.
\end{proof}

\begin{defn}
  We define the sheaf of filtered $\bB_{\pdR}$-algebras
  \begin{equation*}
    \OB_{\pdR}
    :=
    \bB_{\pdR}\widehat{\otimes}_{\bB_{\dR}}\OB_{\dR}.
  \end{equation*}
\end{defn}

\begin{cor}\label{prop:localdescription-of-OBpdR-sheaves-recII}
  With the notation as in Proposition~\ref{prop:localdescription-of-OBdR+-sheaves-recII},
  there is an isomorphism
  \begin{equation*}
    \bB_{\pdR}|_{\widetilde{X}_C}
    \widehat{\otimes}_{k_0}
    k_0[T]
    \isomap
    \OB_{\pdR}|_{\widetilde{X}_C},
    \quad
    T_i
    \mapsto
    \frac{z_i\widehat{\otimes}1-
    1\widehat{\otimes}[z_i^\flat]}{t}
  \end{equation*}
  of sheaves of filtered $\bB_{\pdR}|_{\widetilde{X}_C}$-modules.
\end{cor}

\subsection{Some bimodule structures}

By \cite[Corollary 6.13]{Sch13pAdicHodge}, the sheaf $\OB_{\dR}^+$ carries
a $\bB_{\dR}^+$-linear integrable connection
$\nabla\colon\OB_{\dR}^+\to
\OB_{\dR}^+\widehat{\otimes}_{\nu^{-1}\cO_X}
\nu^{-1}\Omega_X^1$. It satisfies Griffiths transversality,
$\nabla(F_n)\subseteq F_{n+1}\widehat{\otimes}_{\nu^{-1}\cO_X}
\nu^{-1}\Omega_X^1$, and therefore extends continuously to the completed
sheaf $\OB_{\dR}$. The universal property of the relative enveloping algebra
then gives $\OB_{\dR}$ the structure of a sheaf of
$\nu^{-1}\cD$-$\bB_{\dR}$-bimodules.

With the notation of Lemma~\ref{lem:D-local-description} and
Proposition~\ref{prop:localdescription-of-OBdR+-sheaves-recII}, set
$u_j:=z_j\widehat{\otimes}1-1\widehat{\otimes}[z_j^\flat]$, so that
$T_j=u_j/t$. The action is determined locally by
\begin{equation*}
  \partial_i(u_j)=\delta_{ij}
  \quad\text{and}\quad
  \partial_i(T_j)=t^{-1}\delta_{ij},
\end{equation*}
where $\delta_{ij}$ denotes the Kronecker delta.
Since the connection sends $F_m$ into $F_{m+1}$, every
differential operator of order at most $n$ sends $F_m$
into $F_{m+n}$.
Hence the action is bounded on every
finite-order piece $\cD_{\leq n}$.  This shows that $\OB_{\dR}$ is a
sheaf of filtered $\nu^{-1}\cD$-$\bB_{\dR}$-modules.

Finally, the action extends to
$\OB_{\pdR}=\bB_{\pdR}\widehat{\otimes}_{\bB_{\dR}}\OB_{\dR}$ by
extending the connection $\mathbb B_{\mathrm{pdR}}$-linearly and continuously. Consequently,
$\OB_{\pdR}$ is a sheaf of filtered $\nu^{-1}\cD$-$\bB_{\pdR}$-bimodules.

\section{Key technical computations}
\label{sec:OBpdR-is-tilting}

In this section, we establish the key technical inputs for the proof of
Theorem~\ref{introthm:main}.

\subsection{Cohomology computations}

The assertions in this subsection are local on $X$ for the analytic topology.
We may therefore assume that $X=\Sp A$ is affinoid and equipped with an étale
morphism
\begin{equation*}
  X\to\bT^d=\Sp k\left\<z_1^{\pm},\dots,z_d^{\pm}\right\>
\end{equation*}
which factors as a composite of rational embeddings and finite étale maps.
Let $\widetilde{X}\to X$ denote the associated standard toric pro-étale
cover, and write $\widetilde{X}_C:=\widetilde{X}\times_k C$.

\begin{prop}\label{prop:usolid-coho-2}
  The canonical morphism
  \begin{equation*}
    \cO(X)
    \to
    R\Gamma(X_{\proet},\OB_{\pdR})
  \end{equation*}
  is an isomorphism of filtered $k_0$-modules.
\end{prop}

\begin{remark}
    The argument below is partly inspired by the Galois-cohomological arguments of
    Colmez--Gilles--Nizioł for $B_{\pdR}$~\cite[Théorème~1.4]{colmez_gilles_niziol_2025_cst_compact_support}.
\end{remark}

\begin{proof}
  Put $A=\cO(X)$. For $r\in\bZ$ and $N\geq0$, let
  $\cP_{r,N}\subseteq\gr_r\OB_{\pdR}$ be the subsheaf
  of polynomials of degree at most $N$ in $\ell=\log t$.
  This does not depend on the choice of $t$, because changes of $t$ result in translations of $\ell$.
  In particular, we have $\cP_{r,0}=\gr_r\OB_{\dR}$
  and, for $N\geq1$, the exact sequences
  \[
    0\to\cP_{r,N-1}\to\cP_{r,N}
    \stackrel{\beta_{r,N}}{\to}\gr_r\OB_{\dR}\to0
  \]
  where $\beta_{r,N}$ sends a polynomial $\sum_{i=0}^{N}b_i\ell^i$
  to its leading coefficient $b_N$.

  By \cite[Proposition~6.16(ii)]{Sch13pAdicHodge},
  the cohomology of $\gr_r\OB_{\dR}$ vanishes unless
  $r=0$ and $q\in\{0,1\}$, in which cases it is
  $A$ and $A\log\chi_{\cycl}$, respectively.
  
Set $\beta_{0,0}=\mathrm{id}$. For $N\geq1$, the
connecting homomorphism
\[
  \partial_N\colon A\longrightarrow
  H^1(X_{\proet},\cP_{0,N-1})
\]
sends $a$ to the class of the cocycle
\[
  \sigma\longmapsto
  a\bigl((\ell+\log\chi_{\cycl}(\sigma))^N-\ell^N\bigr).
\]
Its image under $H^1(\beta_{0,N-1})$ is
$Na\log\chi_{\cycl}$. Inductively,
$H^1(\beta_{0,N-1})$ is an isomorphism, so $\partial_N$
is an isomorphism. The long exact sequence then shows
that $H^1(\beta_{0,N})$ is an isomorphism as well.
  Consequently,
  $H^0(X_{\proet},\cP_{0,N})=A$,
  $H^1(X_{\proet},\cP_{0,N})\cong A$, and the maps
  \begin{equation*}
    H^1(X_{\proet},\cP_{0,N-1})\to H^1(X_{\proet},\cP_{0,N})
  \end{equation*}
  are zero. All higher cohomology vanishes.
  For $r\ne0$, all cohomology of $\cP_{r,N}$ vanishes.
  Since the affinoid object $X$ is coherent in $X_{\proet}$
  by \cite[Proposition~3.12(iii)]{Sch13pAdicHodge}, its cohomology commutes with filtered colimits.
  This implies
  \begin{equation}\label{eq:prop:usolid-coho-2}
    R\Gamma(X_{\proet},\gr_r\OB_{\pdR})
    =\varinjlim_NR\Gamma(X_{\proet},\cP_{r,N})
    \cong
    \begin{cases}
      A[0], & r=0,\\
      0, & r\ne0.
    \end{cases}
  \end{equation}

  On affinoid perfectoid objects over $\widetilde X_C$,
  each $\gr_r\OB_{\pdR}$ is a direct sum of copies of
  $\widehat{\cO}$ and hence has no higher cohomology by
  \cite[Lemma~4.10(v)]{Sch13pAdicHodge}.
  The same holds for the quotients
  $F_n\OB_{\pdR}/F_{-m}\OB_{\pdR}$, whose transition
  maps on sections are surjective.
  Completeness and \cite[Lemma~3.18]{Sch13pAdicHodge}
  therefore give
  \[
    R\Gamma(X_{\proet},F_n\OB_{\pdR})
    \cong
    R\varprojlim_{m\to\infty}
    R\Gamma\left(
      X_{\proet},
      F_n\OB_{\pdR}/F_{-m}\OB_{\pdR}
    \right).
  \]
  By~\eqref{eq:prop:usolid-coho-2} the complexes on the right are,
  for sufficiently large $m$, canonically $A[0]$ if
  $n\geq0$ and zero otherwise. Their transition maps
  are isomorphisms. Thus
  \[
    R\Gamma(X_{\proet},F_n\OB_{\pdR})
    \cong
    \begin{cases}
      A[0], & n\geq0,\\
      0, & n<0.
    \end{cases}
  \]
  These identifications are induced by the canonical
  map from $A$ with its trivial filtration, proving
  the assertion.
\end{proof}

\subsection{Computing an endomorphism sheaf}

\begin{thm}\label{thm:rho-iso}
  The $\nu^{-1}\cD$-$\bB_{\pdR}$-bimodule structure on $\OB_{\pdR}$
  induces an isomorphism
  \begin{equation*}
    \rho\colon
    \cD
    \isomap
    R\nu_*
    R\shHom_{\bB_{\pdR}}^{F}
    \left(
      \OB_{\pdR},
      \OB_{\pdR}
    \right)
  \end{equation*}
  in $D(\Sh(X,\Fil_k))$.
\end{thm}

The assertion is local on $X$. We may therefore assume that
$X=\Sp A$ admits an \'etale morphism to $\bT^d$ and retain the preceding
notation. Write $F_\bullet$ for the Hodge filtration and set
\begin{equation*}
  \delta_i:=z_i\partial_i,
  \qquad
  \cE:=
  R\shHom_{\bB_{\pdR}}^{F}
  \left(
    \OB_{\pdR},
    \OB_{\pdR}
  \right).
\end{equation*}
The internal Hom carries its usual filtration:
\begin{equation*}
  F_n\shHom_{\bB_{\pdR}}^{F}
  \left(
    \OB_{\pdR},
    \OB_{\pdR}
  \right)
  =
  \left\{
    \Phi\ \middle|\
    \Phi(F_m\OB_{\pdR})
    \subseteq
    F_{m+n}\OB_{\pdR}
    \text{ for every }m
  \right\}.
\end{equation*}
For $\alpha=(\alpha_1,\dots,\alpha_d)\in\bN^d$, write
\begin{equation*}
  |\alpha|:=\alpha_1+\cdots+\alpha_d,
  \qquad
  \alpha!:=\alpha_1!\cdots\alpha_d!,
  \qquad
  \delta^{[\alpha]}
  :=
  \frac{\delta_1^{\alpha_1}\cdots\delta_d^{\alpha_d}}{\alpha!}.
\end{equation*}

\begin{lem}\label{lem:rho-local-model}
  The complex $\cE$ is acyclic on every affinoid perfectoid object over
  $\widetilde{X}_C$. Consequently, the \v{C}ech complex of
  $\widetilde{X}_C\to X$ computes $R\Gamma(X_{\proet},\cE)$.
\end{lem}

\begin{proof}
  Fix an affinoid perfectoid $U$ over $\widetilde{X}_C$. By
  Corollary~\ref{prop:localdescription-of-OBpdR-sheaves-recII}, we have
  \[
    \OB_{\pdR}|_U
    \cong
    \bB_{\pdR}|_U
    \widehat{\otimes}_{k_0}k_0[T].
  \]
As explained in the proof of Proposition~\ref{prop:usolid-coho-2}, we have
\[
  F_n\OB_{\pdR}|_U
  \cong
  R\varprojlim_{m\geq1}
  \left(F_n\OB_{\pdR}/F_{n-m}\OB_{\pdR}\right)|_U.
\]
Thus Lemma~\ref{lem:completion-ghostadjunction-sheaves}
applies and gives the following identification:
  \begin{align*}
    R\shHom_{\bB_{\pdR}|_U}^{F}
    \left(
        \OB_{\pdR}|_U,
      \OB_{\pdR}|_U
    \right)
    &\cong R\shHom_{\bB_{\pdR}|_U}^{F}
    \left(
        \bB_{\pdR}|_U
        \widehat{\otimes}_{k_0}k_0[T],
      \OB_{\pdR}|_U
    \right) \\
    &\cong
    R\shHom_{\bB_{\pdR}|_U}^{F}
    \left(
      \bB_{\pdR}|_U
      \otimes_{k_0}k_0[T],
      \OB_{\pdR}|_U
    \right).
  \end{align*}
  Since each $\gr_r\OB_{\pdR}|_U$ is a direct sum of copies of
  $\widehat{\cO}|_U$, the same d\'evissage and completeness argument as in
  the proof of Proposition~\ref{prop:usolid-coho-2} gives
  \[
    R\Gamma(U,\OB_{\pdR}|_U)
    \cong
    \OB_{\pdR}(U).
  \]
  Consequently, derived tensor--Hom adjunction gives
  \begin{equation*}
    \begin{aligned}
      R\Gamma(U,\cE)
      &\cong
      R\Gamma
      \left(
        U,
        R\shHom_{\bB_{\pdR}|_U}^{F}
        \left(
            \bB_{\pdR}|_U
            \widehat{\otimes}_{k_0}k_0[T],
          \OB_{\pdR}|_U
        \right)
      \right) \\
      &\cong
      R\Gamma
      \left(
        U,
        R\shHom_{\bB_{\pdR}|_U}^{F}
        \left(
          \bB_{\pdR}|_U
          \otimes_{k_0}k_0[T],
          \OB_{\pdR}|_U
        \right)
      \right) \\
      &\cong
      R\intHom_{k_0}^{F}
      \left(
        k_0[T],
        R\Gamma\left(U,\OB_{\pdR}|_U\right)
      \right) \\
      &\cong
      R\intHom_{k_0}^{F}
      \left(
        k_0[T],
        \OB_{\pdR}(U)
      \right).
    \end{aligned}
  \end{equation*}
  Here the third isomorphism also uses the derived adjunction between
  sections over $U$ and extension by zero from $U$. Equivalently, it may be
  verified using an injective resolution, since sections over $U$ of an
  injective filtered sheaf are injective.
  Since $k_0[T]$ carries the trivial filtration, its monomial basis
  identifies it with a direct sum of copies of the projective object $k_0$.
  Hence $k_0[T]$ is projective in $\Fil_{k_0}$, and therefore
  \[
    R\intHom_{k_0}^{F}
    \left(
      k_0[T],
      \OB_{\pdR}(U)
    \right)
    \cong
    \intHom_{k_0}^{F}
    \left(
      k_0[T],
      \OB_{\pdR}(U)
    \right)
  \]
  is concentrated in degree zero. Thus $\cE$ is acyclic on $U$. The same
  argument applies to every iterated fibre product of the cover, so the
  associated \v{C}ech complex computes $R\Gamma(X_{\proet},\cE)$.
\end{proof}

\begin{lem}\label{lem:rho-bounded-taylor}
  On $\widetilde{X}_C$, define
  \begin{equation*}
    W_i
    :=
    \frac{1}{t}
    \log\left(
      \frac{z_i\widehat{\otimes}1}
      {1\widehat{\otimes}[z_i^\flat]}
    \right).
  \end{equation*}
  The change of variables from $T$ to
  $W=(W_1,\dots,W_d)$ is strict for the Hodge filtration. If
  $\gamma_1,\dots,\gamma_d$ are the standard generators of the geometric
  Galois group $\bZ_p(1)^d$, then
  \begin{equation*}
    \gamma_i(W_j)=W_j-\delta_{ij},
    \qquad
    \delta_i(W_j)=t^{-1}\delta_{ij}.
  \end{equation*}
  Hence $Q_i:=t\delta_i$ has filtration degree zero and acts as
  $\partial/\partial W_i$. Put
  \begin{equation*}
    Q^{[\alpha]}
    :=
    \frac{Q_1^{\alpha_1}\cdots Q_d^{\alpha_d}}{\alpha!}.
  \end{equation*}

  On every term $U$ of the \v{C}ech nerve and for every $n\in\bZ$,
  Taylor expansion induces natural strict isomorphisms
  \begin{equation}\label{eq:rho-taylor-delta}
    \begin{aligned}
      F_nR\Gamma(U,\cE)
      &\cong
      \prod_{\alpha\in\bN^d}
      F_n\OB_{\pdR}(U)\,Q^{[\alpha]}\\
      &\cong
      \prod_{\alpha\in\bN^d}
      F_{n-|\alpha|}\OB_{\pdR}(U)\,
      \delta^{[\alpha]}.
    \end{aligned}
  \end{equation}
  In the second description the basis elements
  $\delta^{[\alpha]}$ are defined over $X$; consequently,
  \eqref{eq:rho-taylor-delta} is compatible with the face and degeneracy
  maps of the \v{C}ech nerve coefficientwise.
\end{lem}

\begin{proof}
  The identities
  \begin{equation*}
    W_i
    =
    \frac{1}{t}
    \log\left(
      1+
      \frac{tT_i}{1\widehat{\otimes}[z_i^\flat]}
    \right),
    \qquad
    T_i
    =
    \left(1\widehat{\otimes}[z_i^\flat]\right)
    \frac{\exp(tW_i)-1}{t}
  \end{equation*}
  show that the two changes of variables are mutually inverse and strict.
  The formulas for the actions of $\gamma_i$ and $\delta_i$ follow
  directly from the definitions.

  Every $\Phi\in F_nR\Gamma(U,\cE)$ has a unique Taylor expansion
  \begin{equation*}
    \Phi
    =
    \sum_{\alpha\in\bN^d}
    b_\alpha Q^{[\alpha]},
    \qquad
    b_\alpha\in F_n\OB_{\pdR}(U).
  \end{equation*}
  Conversely, every such family defines an endomorphism: for every
  $f\in\OB_{\pdR}(U)$, the terms $Q^{[\alpha]}(f)$ tend to zero as
  $|\alpha|\to\infty$, so the displayed sum converges in the complete
  filtered module $\OB_{\pdR}(U)$.

  Since
  \[
    Q^{[\alpha]}
    =
    t^{|\alpha|}\delta^{[\alpha]}
  \]
  and multiplication by $t^{|\alpha|}$ identifies
  $F_n\OB_{\pdR}(U)$ with
  $F_{n-|\alpha|}\OB_{\pdR}(U)$, this gives the second isomorphism in
  \eqref{eq:rho-taylor-delta}. Finally, the operators $\delta_i=z_i\partial_i$
  are defined over $X$, so in the second description the face and
  degeneracy maps act only on the coefficients. This proves the asserted
  compatibility with the \v{C}ech nerve.
\end{proof}

We set $B_{\dR}:=\bB_{\dR}(C,\cO_C)$ and $B_{\pdR}:=B_{\dR}[\log t]$.

\begin{lem}\label{lem:rho-affinoid-computation}
  There is a natural strict filtered quasi-isomorphism
  \begin{equation*}
    R\Gamma(X_{\proet},\cE)
    \cong
    \bigoplus_{\alpha\in\bN^d}
    A\delta^{[\alpha]}.
  \end{equation*}
  The filtration on the right is the order filtration:
  \begin{equation*}
    F_n
    \left(
      \bigoplus_{\alpha\in\bN^d}
      A\delta^{[\alpha]}
    \right)
    =
    \bigoplus_{|\alpha|\leq n}
    A\delta^{[\alpha]}.
  \end{equation*}
  In particular, the left-hand side is concentrated in cohomological
  degree zero. The identification is compatible with composition and
  restriction to affinoid subdomains.
\end{lem}

\begin{proof}
  By Lemma~\ref{lem:rho-local-model}, the \v{C}ech complex of
  $\widetilde X_C\to X$ computes $R\Gamma(X_{\proet},\cE)$.
  By Lemma~\ref{lem:rho-bounded-taylor}, its $n$th filtered piece
  identifies coefficientwise with
  \begin{equation*}
    \prod_{\alpha\in\bN^d}
    R\Gamma\left(
      X_{\proet},
      F_{n-|\alpha|}\OB_{\pdR}
    \right)
    \delta^{[\alpha]}.
  \end{equation*}
  Here we use the affinoid-perfectoid acyclicity from the proof of
  Proposition~\ref{prop:usolid-coho-2}; since products are exact at each
  fixed filtration step, the coefficientwise \v{C}ech computation may be
  taken through the product.

  Proposition~\ref{prop:usolid-coho-2} gives
  \begin{equation*}
    R\Gamma\left(
      X_{\proet},
      F_{n-|\alpha|}\OB_{\pdR}
    \right)
    \cong
    \begin{cases}
      A[0], & |\alpha|\leq n,\\
      0, & |\alpha|>n.
    \end{cases}
  \end{equation*}
  Consequently,
  \begin{equation*}
    F_nR\Gamma(X_{\proet},\cE)
    \cong
    \bigoplus_{|\alpha|\leq n}
    A\delta^{[\alpha]}.
  \end{equation*}
  These identifications are compatible as $n$ varies and therefore give
  the asserted strict filtered quasi-isomorphism
  \[
    R\Gamma(X_{\proet},\cE)
    \cong
    \bigoplus_{\alpha\in\bN^d}
    A\delta^{[\alpha]}.
  \]
  Compatibility with composition and restriction follows from the
  functoriality of the Taylor expansion and from the fact that the
  operators $\delta_i$ are defined over $X$.
\end{proof}

\begin{proof}[Proof of Theorem~\ref{thm:rho-iso}]
  Toric affinoids form a basis. On every such affinoid,
  Lemmas~\ref{lem:D-local-description} and
  \ref{lem:rho-affinoid-computation} identify $\rho$ with the strict
  filtered isomorphism
  \begin{equation*}
    \cD(X)
    =
    \bigoplus_{\alpha\in\bN^d}
    A\delta^{[\alpha]}
    \isomap
    R\Gamma(X_{\proet},\cE).
  \end{equation*}
  These identifications respect multiplication and restriction, and
  therefore glue.
\end{proof}

\section{The reconstruction theorem}
\label{sec:reconstruction-recII}

We again fix a smooth rigid-analytic $k$-variety $X$. All filtrations
considered in this section are increasing and exhaustive.

\subsection{Filtered Rees formalism}

As before, we view $\cD$, $\bB_{\pdR}$, and $\OB_{\pdR}$
as sheaves of filtered rings. Using Schneiders' formalism
for quasi-abelian categories, we obtain the derived categories
of filtered modules over $\cD$, $\bB_{\pdR}$, and $\OB_{\pdR}$,
respectively. We denote these categories by
\begin{equation*}
  D_{F}\left(\cD\right), \quad
  D_{F}\left(\bB_{\pdR}\right), \quad
  D_{F}\left(\OB_{\pdR}\right).
\end{equation*}
As explained in \S\ref{subsec:Fil}, these categories are naturally equivalent
to the derived categories of sheaves of graded modules over the corresponding
Rees rings. We denote the latter categories by
\begin{equation*}
  D_{\gr}\left(\sR\cD\right), \quad
  D_{\gr}\left(\sR\bB_{\pdR}\right), \quad
  D_{\gr}\left(\sR\OB_{\pdR}\right).
\end{equation*}

\begin{defn}
  The triangulated category of filtered $\cD$-modules on $X$ with good
  filtrations is the full subcategory
  \begin{equation*}
    D_{F,\coh}^{b}(\cD)
    :=
    D_{\coh,\mathrm{gr}}^{b}(\sR\cD)
    \subset
    D_{\mathrm{gr}}^{b}(\sR\cD)
    \simeq
    D_F^{b}(\cD)
  \end{equation*}
  corresponding to bounded complexes of sheaves of $\sR\cD$-modules
  with coherent cohomology.
\end{defn}

\subsection{An adjunction}

In this subsection, all derived internal Hom and derived Hom functors are
taken in the corresponding filtered derived categories.

Define contravariant functors
\begin{equation*}
\begin{aligned}
  \Sol_{F}\colon
  D_{F}(\cD)^{\op}
  &\to
  D_{F}(\bB_{\pdR}),
  \\
  \cal N^\bullet
  &\mapsto
  R\shHom_{\nu^{-1}\cD}^{F}
  \left(\nu^{-1}\cal N^\bullet,\OB_{\pdR}\right)
\end{aligned}
\end{equation*}
and
\begin{equation*}
\begin{aligned}
  \Rec_{F}\colon
  D_{F}(\bB_{\pdR})^{\op}
  &\to
  D_{F}(\cD),
  \\
  \cal G^\bullet
  &\mapsto
  R\nu_*R\shHom_{\bB_{\pdR}}^{F}
  \left(\cal G^\bullet,\OB_{\pdR}\right).
\end{aligned}
\end{equation*}

Tensor--Hom adjunction gives
\begin{equation}\label{eq:filtered-Sol-Rec-adjunction}
  R\Hom_{\cD}^{F}
  \left(\cal N^\bullet,\Rec_{F}(\cal G^\bullet)\right)
  \cong
  R\Hom_{\bB_{\pdR}}^{F}
  \left(\cal G^\bullet,\Sol_{F}(\cal N^\bullet)\right).
\end{equation}
The identity of $\Sol_{F}(\cal N^\bullet)$ corresponds under
\eqref{eq:filtered-Sol-Rec-adjunction} to the natural morphism
\begin{equation*}
  \rho_{\cal N^\bullet}^{F}\colon
  \cal N^\bullet
  \to
  \Rec_{F}\left(\Sol_{F}(\cal N^\bullet)\right).
\end{equation*}

\subsection{Ingredients for reconstruction}

\subsubsection{The Rees endomorphism identity}

We shall use the following standard compatibility of the derived Rees
equivalence with the two operations appearing above.

\begin{lem}\label{lem:rees-functoriality-recII}
  Let $\cal K^\bullet$ be a complex of filtered sheaves on $X_{\proet}$,
  and let $\cal L^\bullet$ and $\cal M^\bullet$ be complexes of filtered
  modules over a filtered sheaf of rings $\cal A$. There are natural
  isomorphisms
  \begin{equation*}
    \sR\left(R\nu_*\cal K^\bullet\right)
    \cong
    R\nu_*\left(\sR\cal K^\bullet\right)
  \end{equation*}
  and
  \begin{equation*}
    \sR\left(
      R\shHom_{\cal A}^{F}
      \left(\cal L^\bullet,\cal M^\bullet\right)
    \right)
    \cong
    R\shHom^{\mathrm{gr}}_{\sR\cal A}
    \left(\sR\cal L^\bullet,\sR\cal M^\bullet\right),
  \end{equation*}
where the superscript $\gr$ denotes the graded internal Hom; its degree-$n$
component consists of homogeneous morphisms which send degree $p$ to degree
$p+n$.
\end{lem}

\begin{proof}
  Direct image in the category of filtered sheaves is computed
  filtrationwise. Hence, for every $p\in\bZ$, the degree-$p$ component of
  the first morphism is the natural isomorphism
  \begin{equation*}
    \left(
      \sR\left(R\nu_*\cal K^\bullet\right)
    \right)_p
    =
    R\nu_*\left(F_p\cal K^\bullet\right)\hbar^p
    \cong
    \left(
      R\nu_*\left(\sR\cal K^\bullet\right)
    \right)_p.
  \end{equation*}
  These isomorphisms are compatible with multiplication by $\hbar$ and
  therefore assemble to the first isomorphism.

  For the second assertion, recall that the increasing filtration on the
  internal Hom is given by
  \begin{equation*}
    F_n\shHom_{\cal A}^{F}(\cal L,\cal M)
    =
    \left\{
      f\ \middle|\
      f(F_p\cal L)\subseteq F_{p+n}\cal M
      \text{ for every }p\in\bZ
    \right\}.
  \end{equation*}
  Such a morphism $f$ determines a homogeneous $\sR\cal A$-linear
  morphism of degree $n$ by
  \begin{equation*}
    x\hbar^p
    \mapsto
    f(x)\hbar^{p+n}.
  \end{equation*}
  Conversely, every homogeneous morphism of degree $n$ arises in this
  way. Here exhaustivity permits one to choose $p$ such that
  $x\in F_p\cal L$, while compatibility with multiplication by $\hbar$
  shows that the resulting value is independent of this choice. Thus
  there is a natural isomorphism
  \begin{equation*}
    \sR\shHom^{F}_{\cal A}(\cal L,\cal M)
    \cong
    \shHom^{\mathrm{gr}}_{\sR\cal A}
    \left(\sR\cal L,\sR\cal M\right).
  \end{equation*}
  Applying this identification to resolutions computing the derived
  internal Hom gives the second isomorphism.
\end{proof}

\begin{cor}\label{cor:rho-rees-iso}
  The action of $\sR\cD$ on $\sR\OB_{\pdR}$ induces an isomorphism of
  graded algebra objects
  \begin{equation}\label{eq:rho-rees-iso}
    \sR\cD
    \isomap
    R\nu_*
    R\shHom^{\mathrm{gr}}_{\sR\bB_{\pdR}}
    \left(
      \sR\OB_{\pdR},
      \sR\OB_{\pdR}
    \right).
  \end{equation}
\end{cor}

\begin{proof}
  The strict filtered isomorphism of
  Theorem~\ref{thm:rho-iso} becomes \eqref{eq:rho-rees-iso} by
  Lemma~\ref{lem:rees-functoriality-recII}. Since the morphism in
  Theorem~\ref{thm:rho-iso} is induced by the action of $\cD$ on
  $\OB_{\pdR}$, this identification is compatible with multiplication.
\end{proof}

\subsubsection{Perfectness}

An object of $D_{\mathrm{gr}}(\sR\cD)$ is \emph{locally perfect} if it is
locally quasi-isomorphic to a bounded complex of finitely generated
projective graded $\sR\cD$-modules.

\begin{lem}\label{lem:D-finite-global-dimension}
  Let $X=\Sp A$ be smooth of pure dimension $d$. Then $\cD(X)$ and
  $\sR\cD(X)$ are noetherian and
  \begin{equation*}
    \operatorname{l.gldim}\cD(X)\leq2d,
    \qquad
    \operatorname{l.gldim}\sR\cD(X)\leq2d+1.
  \end{equation*}
\end{lem}

\begin{proof}
  The PBW theorem gives
  \begin{equation*}
    \gr\cD(X)
    \cong
    \Sym_A\cT(X).
  \end{equation*}
  Since $A$ is regular of dimension $d$ and $\cT(X)$ is a projective
  $A$-module of rank $d$, the algebra on the right is regular noetherian of
  dimension $2d$. The order filtration is positive and Zariskian, so the
  filtered--graded comparison theorem shows that $\cD(X)$ is noetherian and
  has left global dimension at most $2d$.

  For the Rees algebra, use the auxiliary filtration by the number of
  vector-field factors, with $A[\hbar]$ in degree zero. The PBW theorem
  identifies its associated graded algebra with
  \begin{equation*}
    \Sym_{A[\hbar]}
    \left(\cT(X)\otimes_AA[\hbar]\right),
  \end{equation*}
  which is regular noetherian of dimension $2d+1$. The auxiliary filtration
  is again positive and Zariskian. A second application of filtered--graded
  comparison proves the assertion for $\sR\cD(X)$.
\end{proof}

\begin{prop}\label{prop:coherent-D-is-perfect-recII}
  Every object of $D_{\coh}^{b}(\cD)$ is locally perfect, and every object
  of
  \begin{equation*}
    D_{F,\coh}^{b}(\cD)
    =
    D_{\coh,\mathrm{gr}}^{b}(\sR\cD)
  \end{equation*}
  is locally perfect as an object of $D_{\gr}\left(\sR\cD\right)$.
\end{prop}

\begin{proof}
  The assertion is local on $X$, so let $X=\Sp A$ be smooth of pure
  dimension $d$. If $V=\Sp B$ is an affinoid subdomain of $X$, PBW and
  flatness of $A\to B$ give
  \begin{equation*}
    \cD(V)
    \cong
    B\otimes_A\cD(X),
    \qquad
    \sR\cD(V)
    \cong
    B\otimes_A\sR\cD(X),
  \end{equation*}
  and these restriction rings are flat as right modules over the
  corresponding rings on $X$. After shrinking, a coherent $\cD$-module
  admits a presentation by finite free $\cD$-modules and is obtained by
  localisation from a finitely presented $\cD(X)$-module. Noetherianity
  allows one to continue this presentation to a resolution by finitely
  generated free modules. By Lemma~\ref{lem:D-finite-global-dimension}, the
  $2d$-th syzygy is finitely generated projective, and localisation gives a
  bounded locally projective resolution of the original sheaf.
  The same argument in the graded category, using homogeneous generators
  and grading shifts, shows that the $(2d+1)$-st syzygy is finitely
  generated and projective after forgetting the grading. It is therefore
  graded-projective: the degree-zero component of an ungraded splitting
  of a finite homogeneous free surjection is a graded splitting.
  Thus every coherent graded $\sR\cD$-module has a bounded
  resolution by finitely generated locally projective graded modules. The
  assertion for bounded complexes follows from truncation triangles.
\end{proof}

\subsection{Filtered reconstruction and full faithfulness}

\begin{thm}\label{thm:filtered-reconstruction-recII}
  For every $\cal N^\bullet\in D_{F,\coh}^{b}(\cD)$, the canonical morphism
  \begin{equation*}
    \rho_{\cal N^\bullet}^{F}\colon
    \cal N^\bullet
    \isomap
    \Rec_{F}\left(\Sol_{F}(\cal N^\bullet)\right)
  \end{equation*}
  is an isomorphism.
\end{thm}

\begin{proof}
  Restriction to open subspaces commutes with the functors above and with
  the reconstruction morphism, so the assertion is local on $X$. On a
  sufficiently small affinoid open, the objects $\cal N^\bullet$ for which
  $\rho_{\cal N^\bullet}^{F}$ is an isomorphism form a thick triangulated
  subcategory of $D_{\mathrm{gr}}(\sR\cD)$ that is stable under grading
  shifts. By Corollary~\ref{cor:rho-rees-iso}, this subcategory contains
  $\sR\cD$, and hence it contains every locally perfect graded complex.
  The result now follows from
  Proposition~\ref{prop:coherent-D-is-perfect-recII}.
\end{proof}

\begin{cor}\label{cor:fullyfaithful-filtered-recII}
  The filtered solution functor induces a fully faithful embedding
  \begin{equation*}
    \Sol_{F}\colon
    D_{F,\coh}^{b}(\cD)^{\op}
    \hookrightarrow
    D_F(\bB_{\pdR}).
  \end{equation*}
\end{cor}

\begin{proof}
  Let $\cal M^\bullet,\cal N^\bullet\in D_{F,\coh}^{b}(\cD)$.
  By \eqref{eq:filtered-Sol-Rec-adjunction} and
  Theorem~\ref{thm:filtered-reconstruction-recII}, there are natural
  isomorphisms
  \begin{align*}
    R\Hom_{\bB_{\pdR}}^{F}
    \left(
      \Sol_{F}(\cal M^\bullet),
      \Sol_{F}(\cal N^\bullet)
    \right)
    &\cong
    R\Hom_{\cD}^{F}
    \left(
      \cal N^\bullet,
      \Rec_{F}\left(\Sol_{F}(\cal M^\bullet)\right)
    \right) \\
    &\cong
    R\Hom_{\cD}^{F}
    \left(\cal N^\bullet,\cal M^\bullet\right).
  \end{align*}
  This proves the assertion.
\end{proof}

\subsection{Implicit filtrations}

Finally, we prove variants of Theorem~\ref{thm:filtered-reconstruction-recII}
and Corollary~\ref{cor:fullyfaithful-filtered-recII} where the filtrations are implicit.

For every exhaustively filtered commutative ring $R$, taking the zeroth
filtered piece gives an exact functor
\begin{equation*}
  \Fil_R\to\Mod_{F_0R}, \quad M\mapsto F_0M.
\end{equation*}
Here exactness means that strictly exact sequences of filtered modules
are sent to exact sequences.  The same construction gives an exact functor
\begin{equation*}
  \Fil_{\mathcal R}\to\Mod_{F_0\mathcal R},
  \quad\mathcal F\mapsto F_0\mathcal F
\end{equation*}
for every exhaustively filtered sheaf of commutative rings $\mathcal R$.
In the Rees description this is the degree-zero functor, and
sheafification is computed degreewise. Thus it also induces a functor
between the derived categories. In particular, we obtain
\begin{equation}\label{eq:zeroth-filtered-piece}
  D_F(\bB_{\pdR})\to D(\bB_{\pdR}^{+}),
  \quad\mathcal F^\bullet\mapsto F_0\mathcal F^\bullet.
\end{equation}
Since $F_0$ is exact, no further derivation is required.

\begin{defn}
  The \emph{solution functor} is
  \begin{equation*}
    \Sol\colon
    D_F(\cD)^{\op}
    \to D(\bB_{\pdR}^{+}), \quad
    \cal N^\bullet\mapsto F_0\Sol_F(\cal N^\bullet).
  \end{equation*}
\end{defn}

At the level of modules, extension of scalars gives a functor
\begin{equation}\label{eq:fil-basechange}
  \Mod_{F_0R}\to\Fil_R, \quad
  N\mapsto R\otimes_{F_0R}N,
\end{equation}
where
\begin{equation*}
  F_n\left(R\otimes_{F_0R}N\right)
  :=\im\left(
    F_nR\otimes_{F_0R}N
    \to R\otimes_{F_0R}N
  \right).
\end{equation*}
This functor is left adjoint to $F_0$. Explicitly, a morphism
$N\to F_0M$ extends uniquely to an $R$-linear morphism
$R\otimes_{F_0R}N\to M$, and the extension preserves filtrations
because $F_nR\cdot F_0M\subseteq F_nM$.
Conversely, a filtered morphism restricts along $n\mapsto1\otimes n$
to a morphism $N\to F_0M$.

The functor~\eqref{eq:fil-basechange} need not be exact, even when every
$F_nR$ is flat over $F_0R$.
We therefore use derived filtered extension of scalars.
More precisely, for an exhaustively filtered sheaf of commutative rings
$\mathcal R$, define
$\mathcal R\otimes_{F_0\mathcal R}^{L}\mathcal G^\bullet$ by
\begin{equation*}
  \sR\left(
    \mathcal R\otimes_{F_0\mathcal R}^{L}\mathcal G^\bullet
  \right)
  :=\sR\mathcal R\otimes_{F_0\mathcal R}^{L}\mathcal G^\bullet,
\end{equation*}
where $\mathcal G^\bullet$ is placed in grading degree zero on the
right-hand side.

Graded extension of scalars is left adjoint to taking degree zero.
Transporting its derived adjunction through the Rees equivalence gives
\begin{equation*}
\begin{aligned}
  R\Hom_{\mathcal R}^{F}
  \left(
    \mathcal R\otimes_{F_0\mathcal R}^{L}\mathcal G^\bullet,
    \mathcal F^\bullet
  \right)
  &\cong
  R\Hom_{\sR\mathcal R}^{\gr}
  \left(
    \sR\mathcal R\otimes_{F_0\mathcal R}^{L}\mathcal G^\bullet,
    \sR\mathcal F^\bullet
  \right) \\
  &\cong
  R\Hom_{F_0\mathcal R}
  \left(\mathcal G^\bullet,F_0\mathcal F^\bullet\right).
\end{aligned}
\end{equation*}
In particular, we obtain a functor
\begin{equation}\label{eq:BpdR+-basechange}
  D(\bB_{\pdR}^{+})\to D_F(\bB_{\pdR}), \quad
  \mathcal G^\bullet\mapsto
  \bB_{\pdR}\otimes_{\bB_{\pdR}^{+}}^{L}\mathcal G^\bullet,
\end{equation}
left adjoint to~\eqref{eq:zeroth-filtered-piece}.

\begin{defn}
  The \emph{reconstruction functor} is
  \begin{equation*}
    \Rec\colon
    D(\bB_{\pdR}^{+})^{\op}
    \to D_F(\cD), \quad
    \cal G^\bullet
    \mapsto
    \Rec_F\left(
      \bB_{\pdR}\otimes_{\bB_{\pdR}^{+}}^{L}\cal G^\bullet
    \right).
  \end{equation*}
\end{defn}

Together with~\eqref{eq:filtered-Sol-Rec-adjunction}, the preceding
adjunction yields natural isomorphisms
\begin{equation}\label{eq:unfiltered-Sol-Rec-adjunction}
\begin{aligned}
  R\Hom_{\cD}^{F}
  \left(\cal N^\bullet,\Rec(\cal G^\bullet)\right)
  &\cong
  R\Hom_{\bB_{\pdR}}^{F}
  \left(
    \bB_{\pdR}\otimes_{\bB_{\pdR}^{+}}^{L}\cal G^\bullet,
    \Sol_F(\cal N^\bullet)
  \right) \\
  &\cong
  R\Hom_{\bB_{\pdR}^{+}}
  \left(\cal G^\bullet,F_0\Sol_F(\cal N^\bullet)\right) \\
  &\cong
  R\Hom_{\bB_{\pdR}^{+}}
  \left(\cal G^\bullet,\Sol(\cal N^\bullet)\right).
\end{aligned}
\end{equation}
The identity of $\Sol(\cal N^\bullet)$ corresponds under
\eqref{eq:unfiltered-Sol-Rec-adjunction} to the natural morphism
\begin{equation*}
  \rho_{\cal N^\bullet}\colon
  \cal N^\bullet\to\Rec\left(\Sol(\cal N^\bullet)\right).
\end{equation*}

\begin{lem}\label{lem:forget-fil-recover-fil}
  For every
  $\mathcal F^\bullet\in D_F(\bB_{\pdR})$, the adjunction counit
  \begin{equation*}
    \epsilon_{\mathcal F^\bullet}\colon
    \bB_{\pdR}\otimes_{\bB_{\pdR}^{+}}^{L}F_0\mathcal F^\bullet
    \isomap\mathcal F^\bullet
  \end{equation*}
  is an isomorphism.
\end{lem}

\begin{proof}
  First consider a commutative ring $B^{+}$, a non-unit and non-zero
  divisor $t\in B^{+}$, and $B=B^{+}[1/t]$ with the increasing
  filtration $F_nB=t^{-n}B^{+}$. If $M$ is an exhaustively filtered
  $B$-module, compatibility with multiplication by $t^{-n}$ and $t^n$
  gives, respectively,
  \begin{equation*}
    t^{-n}F_0M\subseteq F_nM,
    \qquad
    t^nF_nM\subseteq F_0M.
  \end{equation*}
  Hence $F_nM=t^{-n}F_0M$ for every $n\in\bZ$, and these
  submodules exhaust $M$.
  Since multiplication by $t$ is invertible on $M$, the submodule
  $F_0M$ is $t$-torsion-free. Consequently, multiplication induces
  an isomorphism
  \begin{equation*}
    B\otimes_{B^{+}}F_0M
    =F_0M[1/t]\isomap M.
  \end{equation*}
  It identifies the $n$th filtered pieces on both sides with
  $t^{-n}F_0M$, and is therefore strict.
  Applying this argument locally to a sheaf $\mathcal F$ of filtered $\bB_{\pdR}$-modules
  yields that the underived adjunction counit
  \begin{equation*}
    \bB_{\pdR}\otimes_{\bB_{\pdR}^{+}}F_0\mathcal F
    \isomap\mathcal F
  \end{equation*}  
  is an isomorphism.

  For the derived assertion, write the Rees ring locally as
  \begin{equation*}
    \sR B
    =\bigoplus_{n\in\bZ}t^{-n}B^{+}\hbar^n
    \cong B^{+}[u,u^{-1}],
    \qquad u=t^{-1}\hbar,\quad\deg u=1.
  \end{equation*}
  For every graded $\sR B$-module $H$, the multiplication map
  \begin{equation*}
    \sR B\otimes_{B^{+}}H_0\to H
  \end{equation*}
  is an isomorphism. Indeed, on the component of degree $n$ its
  inverse sends $x\in H_n$ to $u^n\otimes u^{-n}x$.
  Conversely, for every $B^{+}$-module $N$, the degree-zero component
  of $\sR B\otimes_{B^{+}}N$ is canonically $N$.
  Thus graded extension of scalars and taking degree zero are
  mutually inverse exact equivalences. These statements hold for
  sheaves as well: the multiplication maps are canonical, and their
  local inverses therefore glue independently of the choice of $t$.
  Passing to derived categories and using the Rees equivalence
  proves the assertion. Under these identifications the
  multiplication map is precisely the adjunction counit.
\end{proof}

\begin{thm}\label{thm:unfiltered-reconstruction-recII}
  For every $\cal N^\bullet\in D_{F,\coh}^{b}(\cD)$, the canonical morphism
  \begin{equation*}
    \rho_{\cal N^\bullet}\colon
    \cal N^\bullet
    \isomap\Rec\left(\Sol(\cal N^\bullet)\right)
  \end{equation*}
  is an isomorphism.
\end{thm}

\begin{proof}
  Apply Lemma~\ref{lem:forget-fil-recover-fil} to
  $\mathcal F^\bullet=\Sol_F(\cal N^\bullet)$. Its counit induces an
  isomorphism
  \begin{equation*}
    \Rec_F\left(\Sol_F(\cal N^\bullet)\right)
    \isomap\Rec\left(\Sol(\cal N^\bullet)\right),
  \end{equation*}
  since $\Rec_F$ is contravariant. The construction of
  \eqref{eq:unfiltered-Sol-Rec-adjunction} from the two adjunctions
  shows that
  \begin{equation*}
    \rho_{\cal N^\bullet}
    =\Rec_F\left(\epsilon_{\Sol_F(\cal N^\bullet)}\right)
    \circ\rho_{\cal N^\bullet}^{F}.
  \end{equation*}
  Indeed, under the extension-of-scalars adjunction, the identity of
  $F_0\Sol_F(\cal N^\bullet)$ corresponds to
  $\epsilon_{\Sol_F(\cal N^\bullet)}$; naturality of
  \eqref{eq:filtered-Sol-Rec-adjunction} then gives the displayed
  equality. The first factor is an isomorphism by the lemma, and
  $\rho_{\cal N^\bullet}^{F}$ is an isomorphism by
  Theorem~\ref{thm:filtered-reconstruction-recII}.
\end{proof}

\begin{cor}\label{cor:fullyfaithful-unfiltered-recII}
  The solution functor induces a fully faithful embedding
  \begin{equation*}
    \Sol\colon
    D_{F,\coh}^{b}(\cD)^{\op}
    \hookrightarrow D(\bB_{\pdR}^{+}).
  \end{equation*}
\end{cor}

\begin{proof}
  Let $\cal M^\bullet,\cal N^\bullet\in D_{F,\coh}^{b}(\cD)$.
  By~\eqref{eq:unfiltered-Sol-Rec-adjunction} and
  Theorem~\ref{thm:unfiltered-reconstruction-recII}, there are natural
  isomorphisms
  \begin{align*}
    R\Hom_{\bB_{\pdR}^{+}}
    \left(\Sol(\cal M^\bullet),\Sol(\cal N^\bullet)\right)
    &\cong
    R\Hom_{\cD}^{F}
    \left(
      \cal N^\bullet,
      \Rec\left(\Sol(\cal M^\bullet)\right)
    \right) \\
    &\cong
    R\Hom_{\cD}^{F}
    \left(\cal N^\bullet,\cal M^\bullet\right).
  \end{align*}
  This proves the assertion.
\end{proof}

\section{Application: De Rham Zariski-constructible sheaves}
\label{section:application--dR}

We write $D_{\mathrm{zc}}^b(X_{\et},\bZ_p)$ for the category of
Zariski-constructible complexes of
\cite[Definition 3.32]{BhattHansenZC}, with boundedness understood
locally on $X$.\footnote{Bhatt--Hansen denote this category by
$D_{\mathrm{zc}}^{(b)}(X,\bZ_p)$.}

By~\cite[Remark 3.38]{BhattHansenZC}, the construction with integral
adic coefficients can be carried out using the pro-\'etale topology.
We describe explicitly the realisation on the site $X_{\proet}$ used
in this paper. As explained in~\cite[Remark 3.33]{BhattHansenZC}, an
object $\cal K^\bullet\in D_{\mathrm{zc}}^b(X_{\et},\bZ_p)$ has finite-level
realisations
\begin{equation*}
  \cal K_n^\bullet
  :=\cal K^\bullet\otimes_{\bZ_p}^{L}\bZ/p^n\bZ
  \in D_{\mathrm{zc}}^b(X_{\et},\bZ/p^n\bZ).
\end{equation*}
These come with coherent identifications
\begin{equation*}
  \cal K_{n+1}^\bullet
  \otimes_{\bZ/p^{n+1}\bZ}^{L}\bZ/p^n\bZ
  \cong\cal K_n^\bullet.
\end{equation*}
This compatible system determines $\cal K^\bullet$.
Compatibility and inverse limits of complexes are understood in the corresponding derived $\infty$-categories.
This allows us to define the pro-\'etale realisation by
\begin{equation*}
  \cal K_{\proet}^\bullet
  :=R\varprojlim\nu^{-1}\cal K_n^\bullet
  \in D(X_{\proet},\widehat{\bZ}_p),
\end{equation*}
where
$\widehat{\bZ}_p:=\varprojlim\nu^{-1}\bZ/p^n\bZ$.
Each $\nu^{-1}\cal K_n^\bullet$ is a complex of sheaves of
$\widehat{\bZ}_p$-modules through the reduction map
$\widehat{\bZ}_p\to\nu^{-1}\bZ/p^n\bZ$, and the derived inverse
limit is taken in $D(X_{\proet},\widehat{\bZ}_p)$.
This gives a triangulated functor
\begin{equation*}
  D_{\mathrm{zc}}^b(X_{\et},\bZ_p)
  \to D(X_{\proet},\widehat{\bZ}_p),
  \qquad\cal K^\bullet\mapsto\cal K_{\proet}^\bullet.
\end{equation*}

We use the notion of a lisse $\bZ_p$-sheaf from
\cite[Definition 8.1]{Sch13pAdicHodge}, with the additional
convention that such sheaves are torsion-free.
For a lisse $\bZ_p$-sheaf $\bL=\bL_\bullet$ on $X_{\et}$,
\cite[Proposition 8.2]{Sch13pAdicHodge} shows that its
pro-\'etale realisation is concentrated in degree zero
and identifies it with the usual sheaf
\[
  \bL_{\proet}
  =\widehat{\bL}
  :=\varprojlim\nu^{-1}\bL_n.
\]

We have a triangulated functor
\begin{equation*}
\begin{aligned}
  \SolB_{\dR}^{+}\colon
  D_{F,\coh}^{b}(\cD)^{\op}
  &\to D(\bB_{\dR}^{+}), \\
  \cal M^\bullet
  &\mapsto
  F_0R\shHom_{\nu^{-1}\cD}^{F}
  \left(\nu^{-1}\cal M^\bullet,\OB_{\dR}\right).
\end{aligned}
\end{equation*}

\begin{lem}\label{lem:SolBdR--Sol--comparison}
  For every $\mathcal{M}^{\bullet}\in D_{F,\coh}^{b}(\mathcal{D})$, we have the canonical isomorphism
  \begin{equation}\label{eq:SolBdR--Sol--comparison}
    \SolB_{\dR}^{+}(\cal M^\bullet)
    \widehat\otimes_{\bB_{\dR}^{+}}^{L}\bB_{\pdR}
    \isomap\Sol_{F}(\cal M^\bullet)
  \end{equation}
  in $D_{F}\left(\bB_{\pdR}\right)$.
\end{lem}

\begin{proof}
  By Proposition~\ref{prop:coherent-D-is-perfect-recII},
  $\cal M^\bullet$ is locally perfect in the filtered derived
  category. Both sides of~\eqref{eq:SolBdR--Sol--comparison}
  are compatible with cohomological shifts, distinguished triangles,
  and direct summands. It therefore suffices to check the assertion
  for $\cal M^\bullet=\cD$ and its filtration shifts.

  The proof of Lemma~\ref{lem:forget-fil-recover-fil} applies equally
  to $\bB_{\dR}^{+}\subset\bB_{\dR}$. In particular, multiplication
  induces a canonical isomorphism
  \[
    \bB_{\dR}\otimes_{\bB_{\dR}^{+}}^{L}
    F_0\OB_{\dR}
    \isomap\OB_{\dR}
  \]
  in $D_F(\bB_{\dR})$.
  Associativity of derived extension of scalars consequently gives
  \[
    \begin{aligned}
      F_0\OB_{\dR}
      \widehat{\otimes}_{\bB_{\dR}^{+}}^{L}\bB_{\pdR}
      &\cong
      \left(
        \bB_{\dR}\otimes_{\bB_{\dR}^{+}}^{L}
        F_0\OB_{\dR}
      \right)
      \widehat{\otimes}_{\bB_{\dR}}^{L}\bB_{\pdR}\\
      &\cong
      \OB_{\dR}
      \widehat{\otimes}_{\bB_{\dR}}^{L}\bB_{\pdR}\\
      &=
      \OB_{\pdR}.
    \end{aligned}
  \]
  The last isomorphism follows from
  Lemma~\ref{lem:pdR-flat-dR}, exactness of separated completion,
  and the definition of $\OB_{\pdR}$.
  This shows that~\eqref{eq:SolBdR--Sol--comparison} is an isomorphism
  for $\cal M^\bullet=\cD$.
  Applying the same argument to every filtration shift of
  $\OB_{\dR}$ proves the assertion for every filtration shift
  of $\cD$.
\end{proof}

\begin{defn}\label{def:deRham-zc-complex}
  A complex $\cal K^\bullet\in D_{\mathrm{zc}}^b(X_{\et},\bZ_p)$ is
  \emph{de Rham} if there exist
  $\cal M^\bullet\in D_{F,\coh}^b(\cD)$ and an isomorphism
  \begin{equation}\label{eq:deRham-zc-period-identification}
    \SolB_{\dR}^{+}(\cal M^\bullet)
    \widehat{\otimes}_{\bB_{\dR}^{+}}^{L}
    \bB_{\dR}
    \cong
    \cal K_{\proet}^\bullet
    \widehat\otimes_{\widehat{\bZ}_p}^{L}\bB_{\dR}
  \end{equation}
  in $D_F(\bB_{\dR})$.
\end{defn}

Theorem~\ref{thm:comparison-Scholze-deRham} below shows that
Definition~\ref{def:deRham-zc-complex} recovers the classical notion for
local systems. Its proof uses our main Theorem~\ref{introthm:main}.

First, we need a few lemmata.

\begin{lem}\label{lem:nilpotentoperator}
  There exists a filtered $\bB_{\dR}$-linear derivation $\delta\colon\bB_{\pdR}\to\bB_{\pdR}$
  which is locally given by $d/(d\log t)$ for any choice of Fontaine's $t$.
\end{lem}

\begin{proof}
We construct $\delta$ by descent.

Choose a pro-\'etale cover $\{U_i\to X\}$ on which
$\widehat{\bZ}_p(1)$ admits bases $\epsilon_i$, and put
$t_i=\log[\epsilon_i]$. On $U_{ij}:=U_i\times_XU_j$, write
\[
  \epsilon_j=\epsilon_i^{a_{ij}},
  \qquad
  t_j=a_{ij}t_i,
  \qquad
  a_{ij}\in\widehat{\bZ}_p^\times(U_{ij}).
\]
The corresponding local presentations
$\bB_{\pdR}|_{U_i}=\bB_{\dR}|_{U_i}[\ell_i]$,
where $\ell_i=\log t_i$ is a formal variable, are glued by
\[
  \ell_j=\ell_i+\log(a_{ij}).
\]
Let $\delta_i$ be the $\bB_{\dR}|_{U_i}$-linear derivation determined
by $\delta_i(\ell_i)=1$. Since
$\log(a_{ij})\in\bB_{\dR}(U_{ij})$, we have
\[
  \delta_i(\ell_j)
  =\delta_i\bigl(\ell_i+\log(a_{ij})\bigr)
  =1.
\]
Thus the $\delta_i$ agree on overlaps and define $\delta$.
The same calculation shows that $\delta$ is independent of the
chosen local bases. Since $\ell_i$ has filtration degree zero,
$\delta$ preserves the filtration.

\end{proof}

\begin{lem}\label{lem:nilpotentoperator-ses}
  We have the strict short exact sequence
  \begin{equation*}
    0\to\OB_{\dR}\to\OB_{\pdR}\stackrel{\delta}{\to}\OB_{\pdR}\to0
  \end{equation*}
  of sheaves of filtered $\OB_{\dR}$-modules.
\end{lem}

\begin{proof}
  This follows from a local computation, using
  Proposition~\ref{prop:localdescription-of-OBdR+-sheaves-recII},
  Corollary~\ref{prop:localdescription-of-OBpdR-sheaves-recII},
  and Lemma~\ref{lem:pdR-flat-dR}.
\end{proof}

\begin{construction}\label{construction:delta-trivial-on-M}
  Given $\mathcal{M}^{\bullet}\in D_{F,\coh}^{b}(\cD)$, the derivation $\delta$
  induces an operator
  \[
  \delta_{\Sol}\colon\Sol_F(\cal M^{\bullet})\to\Sol_F(\cal M^{\bullet})
  \]
  locally given by $(\delta_{\Sol}s)(n)=\delta(s(n))$.
  It induces an operator
  \[
    \delta_{\Hom}\colon
    R\shHom_{\bB_{\pdR}}^{F}
    \left(\Sol_F(\cal M^{\bullet}),\OB_{\pdR}\right)
    \to
    R\shHom_{\bB_{\pdR}}^{F}
    \left(\Sol_F(\cal M^{\bullet}),\OB_{\pdR}\right)
  \]
via the formula $\delta_{\Hom}(f):=\delta\circ f-f\circ\delta_{\Sol}$;
see \S\ref{sec:modules-derivations} for details.
We remark that both $\delta_{\Sol}$ and $\delta_{\Hom}$
are not $\mathbb{B}_{\pdR}$-linear, but filtration-preserving $\delta$-connections.

Since $\delta$ commutes with the $\nu^{-1}\cD$-action, applying
$R\nu_*$ gives a $\cD$-linear operator $\delta_{\Rec}$ on
\begin{equation*}
    R\nu_{*}R\shHom_{\bB_{\pdR}}^{F}
    \left(\Sol_F(\cal M^{\bullet}),\OB_{\pdR}\right)
    =\Rec_{F}\left(\Sol_{F}\mathcal{M}^{\bullet}\right)
    \stackrel{\text{\ref{thm:filtered-reconstruction-recII}}}{\cong}\mathcal{M}^{\bullet}.
\end{equation*}
\end{construction}

\begin{lem}\label{lem:delta-trivial-on-M}
  With the notation as in Construction~\ref{construction:delta-trivial-on-M},
  $\delta_{\Rec}=0$.
\end{lem}

\begin{proof}
The reconstruction isomorphism (Theorem~\ref{thm:filtered-reconstruction-recII})
\[
  \rho_{\cal M^\bullet}^{F}\colon
  \cal M^\bullet
  \isomap
  \Rec_F\left(\Sol_F(\cal M^\bullet)\right)
\]
is induced by evaluation. For homogeneous local sections $n$ and $s$ of complexes
computing these objects, write
$\operatorname{ev}_n(s)=(-1)^{|n||s|}s(n)$,
where $|n|$ and $|s|$ denote the cohomological degrees of $n$ and $s$, respectively.
Then
\[
  \begin{aligned}
    (\delta_{\Hom}\operatorname{ev}_n)(s)
    &=
    (-1)^{|n||s|}
    \left(
      \delta(s(n))-(\delta_{\Sol}s)(n)
    \right)\\
    &=0.
  \end{aligned}
\]
Consequently,
$\delta_{\Rec}\circ\rho_{\cal M^\bullet}^{F}=0$.
Since $\rho_{\cal M^\bullet}^{F}$ is an isomorphism, we obtain
$\delta_{\Rec}=0$.
\end{proof}

\begin{lem}\label{lem--thm:comparison-Scholze-deRham}
For a lisse $\bZ_p$-sheaf $\bL$ on $X_{\et}$, put
\[
    D_{\dR}(\bL)
    :=R\nu_*\left(
      \widehat{\bL}\otimes_{\widehat{\bZ}_p}\OB_{\dR}
    \right),
    \qquad
    D_{\pdR}(\bL)
    :=R\nu_*\left(
      \widehat{\bL}\otimes_{\widehat{\bZ}_p}\OB_{\pdR}
    \right),
  \]
  regarded as objects of $D_F(\cD)$ with their natural
  filtrations.
  Write $D_{\dR}^0(\bL):=H^0D_{\dR}(\bL)$ for the zeroth cohomology, equipped with the filtration
  \[
  F_nD_{\dR}^0(\bL)
  :=\nu_*\left(
    \widehat{\bL}\otimes_{\widehat{\bZ}_p}F_n\OB_{\dR}
  \right).
  \]
  If $\delta$ acts trivially on $D_{\pdR}(\bL)$ in
  $D_F(\cD)$, then the canonical morphism
  \[
    D_{\dR}^0(\bL)\isomap D_{\pdR}(\bL)
  \]
  is an isomorphism in $D_F(\cD)$.
\end{lem}

\begin{proof}
We claim that
\[
  R^q\nu_*\left(
    \widehat{\bL}\otimes_{\widehat{\bZ}_p}F_n\OB_{\dR}
  \right)=0
  \qquad(n\in\bZ,\ q\geq2).
\]
Since $F_n\OB_{\dR}\cong F_0\OB_{\dR}(-n)$, Tate twisting
reduces this to $n=0$. Working \'etale locally, we may assume
that $\Gamma:=\operatorname{Gal}(k(\mu_{p^\infty})/k)\cong\bZ_p$.
Put $K:=\widehat{k(\mu_{p^\infty})}$ and choose a topological
generator $\gamma$ of $\Gamma$.

By \cite[Theorem~3.8(i)]{LiuZhuRH2017}
\footnote{Although \cite{LiuZhuRH2017} assumes that $k$ is a
finite extension of $\bQ_p$, the results used here remain valid
for complete discretely valued fields of mixed characteristic
$(0,p)$ with perfect residue field; see
\cite[Notation and conventions]{MR4536903}.}
 and Cartan--Leray,
on a sufficiently small affinoid $U$ we have
\[
  R\Gamma\left(
    U_{\proet},
    \widehat{\bL}\otimes_{\widehat{\bZ}_p}F_0\OB_{\dR}
  \right)
  \cong R\Gamma_{\mathrm{cont}}(\Gamma,M),
  \qquad
  M:=F_0\mathcal{RH}(\bL[1/p])(U_K),
\]
where $\mathcal{RH}$ denotes the geometric Riemann--Hilbert
functor of loc.\ cit.
The module $M$ is the inverse limit of its Banach quotients
$M/t^rM$, with surjective transition maps. The standard
procyclic cohomology calculation gives
\[
  R\Gamma_{\mathrm{cont}}(\Gamma,M)
  \cong
  R\varprojlim_r
  \left[M/t^rM\xrightarrow{\gamma-1}M/t^rM\right]
  \cong
  \left[M\xrightarrow{\gamma-1}M\right].
\]
The last isomorphism follows from the surjectivity of the
transition maps in each degree. This complex is concentrated
in degrees $0$ and $1$, proving the claim.

Since the period sheaves are $\widehat{\bQ}_p=\widehat{\mathbb{Z}}_{p}[1/p]$-modules
and $\widehat{\bL}[1/p]$ is locally free over
$\widehat{\bQ}_p$, tensoring and applying $R\nu_*$
to the strict short exact sequence in
Lemma~\ref{lem:nilpotentoperator-ses} gives, for every $n$,
a distinguished triangle
  \[
    F_nD_{\dR}(\bL)
    \to F_nD_{\pdR}(\bL)
    \xrightarrow{\delta} F_nD_{\pdR}(\bL)
    \to F_nD_{\dR}(\bL)[1].
  \]
  By hypothesis, the middle arrow is zero in the derived
  category. For $q\geq1$, the associated long exact
  sequence contains
  \[
    H^q\left(F_nD_{\pdR}(\bL)\right)
    \xrightarrow{0}
    H^q\left(F_nD_{\pdR}(\bL)\right)
    \to
    H^{q+1}\left(F_nD_{\dR}(\bL)\right)=0.
  \]
  Thus $H^q(F_nD_{\pdR}(\bL))=0$ for $q\geq1$.
  Negative cohomology also vanishes, since
  $F_nD_{\pdR}(\bL)$ is the derived direct image of a sheaf.

  The beginning of the same long exact sequence now gives
  the canonical isomorphism
  \[
    F_nD_{\dR}^0(\bL)
    =
    H^0\left(F_nD_{\dR}(\bL)\right)
    \xrightarrow{\sim}
    H^0\left(F_nD_{\pdR}(\bL)\right).
  \]
  Consequently, the canonical comparison is a
  quasi-isomorphism on every filtration piece.
  These comparisons are compatible with the filtration
  maps and the $\cD$-action, proving the asserted
  isomorphism in $D_F(\cD)$.
\end{proof}

\begin{thm}\label{thm:comparison-Scholze-deRham}
  Let $\bL$ be a lisse $\bZ_p$-sheaf on $X_{\et}$. The following are equivalent:
  \begin{itemize}
    \item[(i)]
      $\bL$ is de Rham in the sense
      of~\cite[Definition 8.3]{Sch13pAdicHodge}.
    \item[(ii)]
      $\bL$ is de Rham in the sense of
      Definition~\ref{def:deRham-zc-complex}.
  \end{itemize}
\end{thm}

\begin{proof}
  Suppose first that $\bL$ is de Rham in the sense
  of~\cite[Definition 8.3]{Sch13pAdicHodge}. There exist a filtered
  vector bundle $\cE$ with integrable connection satisfying Griffiths
  transversality and a filtered isomorphism
  \begin{equation*}
    F_0\left(
      \nu^{-1}\cE
      \widehat\otimes_{\nu^{-1}\cO}\OB_{\dR}
    \right)^{\nabla=0}
    \otimes_{\bB_{\dR}^{+}}\bB_{\dR}
    \cong
    \widehat{\bL}\otimes_{\widehat{\bZ}_p}\bB_{\dR}.
  \end{equation*}
  The filtered Spencer resolution and the Poincar\'e lemma identify
  the left-hand side with
  $\SolB_{\dR}^{+}(\cE^\vee)
    \widehat{\otimes}_{\bB_{\dR}^{+}}^{L}
    \bB_{\dR}$, where
  $\cE^\vee$ carries the dual filtration. This gives the isomorphism
  required in Definition~\ref{def:deRham-zc-complex}, with
  $\cal M^\bullet=\cE^\vee$.

    Conversely, suppose that $\bL$ is de Rham in the sense of
  Definition~\ref{def:deRham-zc-complex}. Then there exist
  $\cal M^\bullet\in D_{F,\coh}^b(\cD)$ and an isomorphism
  \begin{equation}\label{eq:dR-etale-comparison}
    \SolB_{\dR}^{+}(\cal M^\bullet)\widehat{\otimes}_{\bB_{\dR}^{+}}^L\bB_{\dR}
    \cong
    \widehat{\bL}\widehat{\otimes}_{\widehat{\bZ}_p}^L\bB_{\dR}
  \end{equation}
  in $D_F(\bB_{\dR})$. After extending scalars to $\bB_{\pdR}$,
  Lemma~\ref{lem:SolBdR--Sol--comparison} and
  Theorem~\ref{thm:filtered-reconstruction-recII} give
  \begin{equation*}
    \cal M^\bullet
    \cong
    \Rec_{F}\left(\Sol_{F}(\cal M^\bullet)\right)
    \cong
    \Rec_{F}\left(
      \widehat{\bL}
      \widehat\otimes_{\widehat{\bZ}_p}^L
      \bB_{\pdR}
    \right)
    \cong
    D_{\pdR}\left(\bL^\vee\right).
  \end{equation*}
  These isomorphisms identify the endomorphism 
  $\delta$ on $D_{\pdR}\left(\mathbb{L}^{\vee}\right)$
  and $\delta_{\Rec}$ on $\mathcal{M}^{\bullet}$
  as in Construction~\ref{construction:delta-trivial-on-M}.
  Indeed,~\eqref{eq:dR-etale-comparison} is defined over
  $\bB_{\dR}$, which $\delta$ annihilates.
  Its extension to
  $\bB_{\pdR}$ is therefore compatible with the coefficient
  $\delta$-connections, and hence with the induced commutator connections on $\Rec_{F}\left(\Sol_{F}(\cal M^\bullet)\right)$.
  
  It follows from Lemma~\ref{lem:delta-trivial-on-M} that the endomorphism
  $\delta$ on $D_{\pdR}\left(\mathbb{L}^{\vee}\right)$ is trivial.
  Thus Lemma~\ref{lem--thm:comparison-Scholze-deRham} applies, giving the isomorphism
  \begin{equation*}
    \cal M^\bullet
    \cong
    D_{\pdR}\left(\mathbb{L}^\vee\right)
    \cong
    D_{\dR}^{0}\left(\mathbb{L}^\vee\right)
  \end{equation*}
  in $D_{F}(\mathcal{D})$.
By \cite[Theorem~3.9(i)]{LiuZhuRH2017}, $\mathcal{M}^{\bullet}$
is represented by a vector bundle $\mathcal{M}$
with integrable connection, concentrated in degree zero.
Forgetting filtrations in~\eqref{eq:dR-etale-comparison}
and taking zeroth cohomology, the Spencer resolution gives
\[
  \left(
    \nu^{-1}\mathcal{M}^\vee
    \widehat\otimes_{\nu^{-1}\cO}\OB_{\dR}
  \right)^{\nabla=0}
  \cong
  \widehat{\bL}\otimes_{\widehat{\bZ}_p}\bB_{\dR}.
\]
By \cite[Theorem~7.6(ii)]{Sch13pAdicHodge}, applied to
$\cM^\vee$ with its trivial filtration, the left-hand side
is a $\bB_{\dR}$-local system of rank $\operatorname{rk}\cM$;
see also \cite[Remark~2.2.11]{MR4536903}. Consequently,
\[
  \operatorname{rk}\cM=\operatorname{rk}\bL[1/p].
\]
For every classical point $x$, pullback compatibility
in \cite[Theorem~3.9(ii)]{LiuZhuRH2017} identifies
$\cM_x$ with the de Rham realisation of
$\bL_x^\vee[1/p]$. The rank equality therefore shows that
$\bL_x^\vee[1/p]$, and hence $\bL_x[1/p]$, is de Rham.
Applying \cite[Theorem~3.9(iii),(iv)]{LiuZhuRH2017}
on each connected component proves the assertion.
\end{proof}

\appendix

\section{Derivations and derived Hom}\label{sec:modules-derivations}

Let $R$ be a commutative ring equipped with a derivation $\delta\colon R\to R$.
An \emph{$R$-module with $\delta$-connection} is an $R$-module $M$,
equipped with an additive map $\delta_{M}\colon M\to M$ satisfying the Leibniz rule
\begin{equation*}
  \delta_{M}(rm)=\delta(r)m+r\delta_{M}(m)
\end{equation*}
for all $r\in R$, $m\in M$.

Equivalently, $M$ is a left module over the \emph{skew-Ore extension}
$R[\partial;\delta]$ of $R$ which is an overring of $R$ that as a left $R$-module is
free on the symbols $\left\{\partial^{n}\colon n\geq0\right\}$ and satisfies
$\partial r-r\partial=\delta(r)$ and $\partial^{i}\partial^{j}=\partial^{i+j}$ for all $i,j\geq0$.

We denote the derived category of left $R[\partial;\delta]$-modules by $D_{\delta}(R)$.

\begin{lem}\label{lem:deltaconn-derivedcategory}
  For any $M^{\bullet},N^{\bullet}\in D_{\delta}(R)$, consider the derived homomorphisms
  $H^\bullet:=R\Hom_{R}(M^{\bullet},N^{\bullet})\in D(R)$ after forgetting the $\delta$-connections. Then,
  $H^\bullet$ naturally lifts to an object in $D_{\delta}(R)$.
\end{lem}

\begin{proof}
  Choose a $K$-injective resolution $N^\bullet\to I^\bullet$
  in the category of $R[\partial;\delta]$-modules. Since $R[\partial;\delta]$ is free as a
  right $R$-module, extension of scalars $R[\partial;\delta]\otimes_R-$
  is exact. Its right adjoint, restriction of scalars,
  therefore preserves $K$-injective complexes. Hence
  \[
    R\operatorname{Hom}_R(M^\bullet,N^\bullet)
    \cong
    \operatorname{Hom}_R(M^\bullet,I^\bullet).
  \]
  Define the action of $\partial$ on the latter complex by
  \[
    \partial(f):=\delta_I\circ f-f\circ\delta_M.
  \]
  The Leibniz rules imply that $\partial(f)$ is $R$-linear
  and that
  \[
    \partial(rf)=\delta(r)f+r\partial(f).
  \]
  Moreover, $\partial$ commutes with the Hom differential.
  Thus this complex is naturally a complex of $R[\partial;\delta]$-modules.
  The construction is independent of the resolution
  and functorial in the derived category.
\end{proof}

The same argument applies to sheaves of graded commutative
rings and derived internal Hom, for a derivation of degree
zero. Applying this construction
to Rees modules gives the filtered version whenever the
derivations preserve the filtrations.

\begingroup
\sloppy
\bibliographystyle{amsplain}
\bibliography{D-modules_reconstruction}
\endgroup

\end{document}